\documentclass{amsart}
\usepackage[mathscr]{eucal}
\usepackage{graphicx}
\usepackage{amscd}
\usepackage{amsmath}
\usepackage{amsthm}
\usepackage{amsxtra}
\usepackage{mathtools}
\usepackage{calc} 
\usepackage{amsfonts}
\usepackage{amssymb}
\usepackage{latexsym}
\usepackage{pdfsync}
\usepackage[all]{xy}

\usepackage[colorlinks, bookmarks=true]{hyperref}

\newtheorem{theorem}{Theorem}[section]
\theoremstyle{plain}

\newtheorem{corollary}[theorem]{Corollary}
\newtheorem{corollary-definition}[theorem]{Corollary-Definition}

\newtheorem{definition}[theorem]{Definition}

\newtheorem{lemma}[theorem]{Lemma}

\newtheorem{proposition}[theorem]{Proposition}

\numberwithin{equation}{section}

\newcommand{\blank}{\hspace{0.04cm} \rule{2.4mm}{.4pt} \hspace{0.04cm} }

\DeclareMathOperator{\ot}{\overset{\rightharpoondown}{\otimes}}
\DeclareMathOperator{\otleft}{\overset{\leftharpoondown}{\otimes}}
\DeclareMathOperator{\oh}{\overline{\mathrm{Hom}}}

\theoremstyle{definition}
\newtheorem{example}[theorem]{Example}
\newtheorem{remark}[theorem]{Remark}

\newcommand{\Ker}{\mathrm{Ker}\,}

\newcommand{\Hom}{\mathrm{Hom}\,}
\DeclareMathOperator{\Ext}{\mathrm{Ext}}
\newcommand{\Tor}{\mathrm{Tor}}

\newcommand{\Z}{\mathbb{Z}\,}
\newcommand{\Q}{\mathbb{Q}\,}

\newcommand{\ev}{\textrm{ev}}
\newcommand{\Tr}{\mathrm{Tr}\,}

\newcommand{\lra}{\longrightarrow}

\newcommand{\Mod}{\mathrm{Mod}\,}
\newcommand{\injt}{\mathfrak{s}}
\newcommand{\injc}{\mathfrak{q}}

\newcommand{\fp}{\mathrm{fp}}

\newcommand{\ab}{\mathrm{Ab}\,}

\usepackage{pgfpages}
\usepackage{tikz}
\usetikzlibrary{matrix,arrows,decorations,backgrounds}

\newcounter{hours}
\newcounter{minutes}

\begin{document}
%
%

\title[]{Injective stabilization of additive functors. II. (Co)torsion and the Auslander-Gruson-Jensen functor}

\author{Alex Martsinkovsky}
\address{Mathematics Department\\
Northeastern University\\
Boston, MA 02115, USA}
\email{a.martsinkovsky@northeastern.edu}
\author{Jeremy Russell}
\address{Department of Mathematics\\
Rowan University\\
201 Mullica Hill Rd, Glassboro, NJ 08028, USA}
\email{russelljj@rowan.edu}

\thanks{The first author was supported in part by the Shota Rustaveli National Science Foundation of Georgia Grant NFR-18-10849}
\date{\today, \setcounter{hours}{\time/60} \setcounter{minutes}{\time-\value
{hours}*60} \thehours\,h\ \theminutes\,min}
\dedicatory{In memory of Gena Puninski}

\begin{abstract}
The formalism of injective stabilization of additive functors is used to define a new notion of the torsion submodule of a module. It applies to arbitrary modules over arbitrary rings. For arbitrary modules over commutative domains it coincides with the classical torsion, and for finitely presented modules over arbitrary rings it coincides with the \textcolor{purple}{Bass torsion}. A formally dual approach -- based on projective stabilization --  gives rise to a new concept: the cotorsion quotient module of a module. This is done  in complete generality -- the new concept is defined for any module over any ring. Unlike torsion, cotorsion does not have classical prototypes.

General properties of these constructs are established. It is shown that  the Auslander-Gruson-Jensen functor applied to the cotorsion functor returns the torsion functor. \textcolor{purple}{As a consequence, a ring is one-sided absolutely pure if and only if each pure injective on the other side is cotorsion-free.} If the injective envelope of the ring is finitely presented, then the right adjoint of the Auslander-Gruson-Jensen functor applied to the torsion functor returns the cotorsion functor. This correspondence establishes a duality between torsion and cotorsion over such rings. In particular, this duality applies to artin algebras. \textcolor{purple}{It is also shown that, over any ring, the character module of the torsion of a module is isomorphic to the cotorsion of the character module of the module. Under various finiteness conditions on the injective envelope of the ring, the derived functors of torsion and cotorsion are computed.}
\end{abstract}

\maketitle
\tableofcontents

\section{Introduction}

This paper is second in a series of three dealing with the injective stabilization of the tensor product and other additive functors. Part~I~\cite{MR-1} provided basic results on this construct, and now we give first applications --  definitions of the torsion submodule of a module and of the cotorsion quotient module of a module. One of the main points here is that this is done in complete generality, without any restrictions on rings or modules. For ease of reference, we sporadically (and somewhat inconsistently) refer to the new torsion as the ``injective torsion''; more often though we simply use the term ``torsion''.

In Section~\ref{S:torsion}, we define the torsion submodule of a module as the injective stabilization of the tensor product with the module, evaluated at the ring \textcolor{purple}{viewed as a module on the other side}, and establish a dozen or so expected properties of this operation. \textcolor{purple}{This is followed up by} definitions of torsion modules and of torsion-free modules. Over commutative domains the injective torsion coincides with the classical torsion, and for finitely presented modules over arbitrary rings it coincides with the Bass torsion, defined at each component as the kernel of the canonical map from the module to its bidual~\cite{B-60}. While the Bass torsion is defined over arbitrary rings and, for finitely presented (or, even, finite) modules over commutative domains, coincides with the classical torsion, this is no longer true for infinite modules. Thus the injective torsion is in general different from  the Bass torsion. However, the former is always contained in the latter, and it turns out that the injective torsion functor commutes with filtered colimits. In fact, it is the largest subfunctor of the Bass torsion functor commuting with \textcolor{purple}{filtered} colimits. As an application, we show that the injective torsion is a radical: the torsion of the module modulo its torsion submodule is zero. If the injective envelope of the ring is flat, then the torsion subfunctor is idempotent, and the torsion-free class is closed under extensions. \textcolor{purple}{This need no longer be true without the flatness assumption (see Remark~\ref{R:low-torsion}). Said differently, over a general ring, there is no torsion theory for which the class of torsion modules is the corresponding torsion class. As a consequence of the functorial definition of torsion, we show that over rings for which torsion splits off, it cannot split off functorially, unless it is trivial.}

In Section~\ref{S:cotorsion} we introduce the notions of the cotorsion quotient module of a module and of the cotorsion-free submodule of a module. \textcolor{purple}{This leads to the corresponding notions of the cotorsion class and the cotorsion-free class. The reader is warned that these classes have little, if anything, to do with the notions of cotorsion theories (or cotorsion pairs). The underlying philosophy in this paper is that cotorsion should be a notion ``dual'' to torsion and \texttt{should thus be a coradical}, i.e., a radical on the opposite category. Notice that the categorical dual of a torsion theory in the sense of Dickson is again a torsion theory. The injective torsion is not a torsion theory, but a simple reversal of the arrows (this amounts to passing to the opposite category) produces just a ``mirror image'' of the torsion and doesn't seem to give any new insights coming from the opposite category. However, there is a more sophisticated correspondence, based on the comparison between the injective stabilization of the tensor product and the projective stabilization of the contravariant Hom functor, that indeed yields a new construct.} Our concept of cotorsion is developed in similarity with our definition of torsion, except that we were not able to find any prior attempts, even over commutative domains, to define \texttt{the cotorsion module of a module}. There have been attempts to define cotorsion modules and cotorsion-free modules in various degrees of generality. However, a quick review of the literature reveals considerable variation in the end results. We remark that, for some of the prior definitions, the cotorsion class contains both cotorsion and cotorsion-free modules in our sense; \textcolor{purple}{see Remarks~\ref{R:cot-Matlis} and~\ref{R:EJ} for details. Similar to torsion, we also show that over rings for which cotorsion splits off it cannot split off functorially, unless it is trivial.}
 
In Section~\ref{S:duality}, we  show that the functor discovered (independently) by Gruson-Jensen and Auslander \textcolor{purple}{sends cotorsion into torsion, thereby converting a metamathematical analogy between cotorsion and injective torsion into a mathematical statement. As a consequence, we show that a ring is absolutely pure on one side if and only if every pure injective on the other side is cotorsion-free.}

\textcolor{purple}{For algebras over commutative rings, we utilize the extension of the Auslander-Reiten formula proved in Part I~\cite[Proposition 9.7 and Remark 9.8]{MR-1}, 
\[
 D_{\mathbf{J}}(A\ot B) \simeq \overline{\Hom} (B, D_{\mathbf{J}}(A)), 
\]
which holds for arbitrary modules $A$ and $B$, to provide additional connections between torsion and cotorsion. Here the ring is an algebra over a commutative coefficient ring $R$, $\mathbf{J}$ is an injective $R$-module, and $D_{\mathbf{J}} := \Hom_{R} (\blank, \mathbf{J})$. In particular, specializing to integer coefficients and $B := \Lambda$, we have that the character module of the torsion of $A$ is isomorphic to the cotorsion of the character module of $A$: 
\[
\injt(A)^{+} \simeq \injc(A^{+})
\]
}

\textcolor{purple}{In Section 5, we look at the properties of torsion and cotorsion under various finiteness assumptions on the injective envelope of the ring. This allows to completely determine all derived functors of torsion and cotorsion.} In the case the injective envelope of the ring is finitely presented, we show that the right adjoint of the Auslander-Gruson-Jensen functor converts torsion back into cotorsion, thus establishing a duality between the two concepts. In particular, this duality holds over artin algebras.

The terminology and notation used in this part are the same as in the first paper~\cite{MR-1} of the series.

\textcolor{purple}{The authors are grateful to the referee for their pointed comments and encouraging suggestions, as well as shared insights. As a result, the authors have made a significant effort to make the paper more ``user-friendly''.}

\section{Torsion over arbitrary rings}\label{S:torsion}

In this section, our goal is to extend the classical torsion, originally defined for modules over commutative domains, to arbitrary modules over arbitrary rings. Let $\Lambda$ be a ring and $A$ a right $\Lambda$-module. We return to the injective stabilization $A \ot \blank$ of $A \otimes \blank$ and want to take a closer look at its $\Lambda$-component. \textcolor{purple}{Recall that the injective stabilization $\overline{F}$ of an additive covariant functor $F$ on a module category (or, more generally, on an abelian category with enough injectives) is the kernel of the natural transformation $F \to R^{0}F$, where $R^{0}F$ is the zeroth right-derived functor of $F$. 
Thus we have a defining exact sequence 
\[
0 \lra \overline{F} \lra F \lra R^{0}F.
\]
Equivalently, $\overline{F}$ can be defined as the largest subfunctor of $F$ vanishing on injectives. Given a module $B$, the value of 
$\overline{F}$ on $B$ can also be determined as follows. Take the injective envelope $\iota : B \to I$, apply $F$, and take the kernel of $F(\iota)$. Thus we have yet another defining exact sequence
\[
0 \lra \overline{F}(B) \lra F(B) \overset{F(\iota)}\lra F(I).
\]
In the case $F := A \otimes \blank$, this becomes
\[
0 \lra A \ot B \lra A \otimes B \lra A \otimes I.
\]
}
Now recall that if $\Lambda$ is a commutative domain and $K$ is its field of fractions, then the classical torsion of $A$ coincides with the kernel of the canonical \textcolor{purple}{localization} map $A \to A \otimes K$. But $K$, being divisible and torsion-free, is also the injective envelope of $\Lambda$, and therefore the kernel  \textcolor{purple}{of the localization map} is isomorphic to the injective stabilization of the functor $A \otimes \blank$ evaluated at~$\Lambda$. This observation leads to the following definition.

\begin{definition}
The (injective) \textit{torsion} of the right $\Lambda$-module $A$ is defined by
\[
\injt(A) : =  (A \ot \blank)(\Lambda) = A \ot \Lambda.
\] 
\end{definition}

Immediately from this definition we have
\begin{proposition}
 If $\Lambda$ is a commutative domain, then the injective torsion coincides with the classical torsion. \qed
\end{proposition}

\textcolor{purple}{Since $\Lambda$ is a bimodule, $\injt(A)$ is a right $\Lambda$-module.} By~\cite[Remark~9.6]{MR-1}, the injective stabilization of the tensor product is a bifunctor. Together with the canonical isomorphism $A \otimes \Lambda \cong A$ and the fact that, in each component, the inclusion map from the injective stabilization to the \textcolor{purple}{functor} is a homomorphism of right $\Lambda$-modules~\cite[Remark~4.2]{MR-1}, this yields

\begin{proposition}\label{P:subfunctor}
$\injt$ is a subfunctor of the identity functor on the category 
 $\mathrm{Mod}$-$\Lambda$ of right $\Lambda$-modules,  \textcolor{purple}{and therefore a subfunctor of the forgetful functor on $\mathrm{Mod}$-$\Lambda$ to abelian groups.}
\qed
\end{proposition}

\begin{corollary}\label{C:s-pres-mono}
 $\injt$ preserves monomorphisms.\footnote{\textcolor{purple}{Any subfunctor of the identity functor preserves monomorphisms.}} In particular, if $\injt(A) = {0}$ and $B$ is a submodule of $A$, then $\injt(B) = {0}$. \qed
\end{corollary}

\begin{corollary}\label{C:inj-stab-s=0}
 \textcolor{purple}{The injective stabilization of $\injt$ is zero: $\overline{\injt} = 0$.}
\end{corollary}

\begin{proof}
 \textcolor{purple}{By~\cite[Lemma 4.8]{MR-1}, the injective stabilization of a functor is zero if and only if it preserves monomorphisms.}
\end{proof}

\textcolor{purple}{If we want to write the definition of torsion as a functor, without referring to the argument, we already have a notation in place: $\injt = \blank \ot \Lambda$. The same cannot be said about our vocabulary since we do not have a succinct name for the right-hand side. Notice that in the case of the tensor product bifunctor, the functors $A \otimes \blank$ and $\blank \otimes B$ are of the same kind and, understandably, are both referred to as the (univariate) tensor product. The situation changes dramatically when dealing with the injective stabilization of the tensor product because the functors $A \ot \blank$ and $\blank \ot B$ are vastly different in their properties. For example, while the former vanishes on injectives, the latter vanishes on projectives (because tensoring with a projective is an exact functor). This need for a nomenclature prompts the following}

\begin{definition}\label{D:tensor-active-vs-inert}
\textcolor{purple}{Let $A$ be a right $\Lambda$-module and $B$ a left $\Lambda$-module. The functor $A \ot \blank$ will be called the \texttt{active injective stabilization} of the tensor product, whereas the functor $\blank \ot B$ will be called the \texttt{inert injective stabilization} of the tensor product.}\footnote{\textcolor{purple}{The more adventurous reader may reverse the direction of the harpoon and look at the bifunctor $A \otleft B$, which gives rise to another pair of active and inert univariate functors.}}
\end{definition}

\textcolor{purple}{Notice that, in each of the two cases above, the name of the functor corresponds to the appellation (active vs inert) of its argument (cf.~\cite[p. 455]{MR-1}). In summary, \texttt{the injective torsion functor is just the inert injective stabilization of the tensor product with~$\prescript{}{\Lambda}{\Lambda}$.}}

\begin{definition}
 The submodule $\injt(A)$ will be called the torsion submodule of $A$. 
\end{definition}

\begin{remark}
\textcolor{purple}{As part of Proposition~\ref{P:subfunctor}, we have that $\injt(A)$ is a submodule of $A$. Mimicking the definition in the commutative case, one can define the torsion $t(A)$ of $A$ as the set of all elements of $A$ which can be annihilated by regular elements of $\Lambda$. It is known~\cite[10.19]{Lam} that $\Lambda$ is right Ore if and only if for each right $\Lambda$-module~$A$, $t(A)$ is a submodule of $A$. It now follows that $\injt \neq t$ if $\Lambda$ is not right Ore.}
\end{remark}

There is another possible candidate, $\mathfrak{t}(A)$, for the torsion submodule of~$A$, first considered by Bass~\cite{B-60}. It is defined by the exact sequence
\[
0 \lra \mathfrak{t}(A) \lra A \overset{e_{A}}\lra A^{\ast\ast}
\]
where $A^{\ast} := \Hom_{\Lambda}(A, \Lambda)$ and $e_{A}$ is the canonical evaluation map. We shall call $\mathfrak{t}(A)$ the \textcolor{purple}{Bass torsion submodule} of~$A$.\footnote{The case of a vanishing Bass torsion (\textcolor{purple}{without using this name}) was investigated in~\cite{MS}; it provides a generalization and a conceptual framework for the notion of linkage of algebraic varieties. That approach was extended to functors by the second author in~\cite{R-15}. The case of a non-vanishing Bass torsion of finitely presented modules over semiperfect rings was studied in~\cite{M-10} \textcolor{purple}{(where it was called 1-torsion)}.} 
Clearly, $\mathfrak{t}$ is also a subfunctor of the identity functor. The Bass torsion is defined for any module and, for finitely generated modules over commutative domains, it coincides with the classical torsion, i.e., with the injective torsion. However, this is no longer true for infinitely generated modules. In fact, as the next example shows, the two could be at the extreme ends in terms of size.  

\begin{example}\label{E:extremes}
Let $\Lambda := \Z$ and $A := \Q$. Since $\Q$ is divisible, 
$\Q^{\ast\ast}= \{0\}$, and therefore $\mathfrak{t}(\Q) = \Q$. On the other hand, $\Q$ is flat and, by~\cite[Proposition~7.2]{MR-1},
$\injt(\Q) = \{0\}$.
\end{example}

While the injective torsion $\injt$ doesn't have this drawback and is indeed our choice for torsion, the Bass torsion functor $\mathfrak{t}$ will also play an important role in our arguments, so now we want to clarify the relationship between the two. The next result puts Example~\ref{E:extremes} in a conceptual framework.

\begin{proposition}\label{P:s-is-a-sub-of-t}
$\injt \subseteq \mathfrak{t}$, i.e., the injective torsion is a subfunctor of the Bass torsion. 
\end{proposition}

\begin{proof}
 Let $A$ be an arbitrary module. Evaluating the natural transformation 
\[
A \otimes \blank \overset{\mu_{A}}\lra (A^{\ast}, \blank)
\]
(see Lemma~\ref{L:mu-on-proj} below) on the injective envelope $\iota : \Lambda \to I$ of $\Lambda$, we have a commutative square
\begin{equation}\label{Eq:mu-square}
\begin{gathered}
 \xymatrix
	{
	A \otimes \Lambda \ar[d]_{\mu_{A}(\Lambda)} \ar[r]^{1 \otimes \iota}
	& A \otimes I \ar[d]^{\mu_{A}(I)}
\\
	(A^{\ast}, \Lambda) \ar@{}[r]^(.30){}="a"^(.8){}="b" \ar@{>->}_{(1, \iota)} 	"a";"b"
	& (A^{\ast}, I).
	}
\end{gathered}
\end{equation}
Here $\injt(A) = \Ker (1 \otimes \iota)$ and $\mathfrak{t}(A) = 
\textcolor{purple}{\Ker \mu_{A}(\Lambda)}$. By the left-exactness of the Hom functor, $(1, \iota)$ is monic. The desired result now follows from the commutativity of the square.
\end{proof}

\begin{corollary}\label{C:Lambda-power}
 $\injt(\Lambda^{N}) = \{0\}$ for any indexing set $N$. \qed
\end{corollary}

Next we want to show that the inclusion of functors $\injt \subseteq \mathfrak{t}$ becomes an equality when restricted to finitely presented modules. This could be done by appealing to prior results (in particular,~\cite[Proposition~9.4]{MR-1}). Instead, we opt for a short, self-contained, and more direct proof, presented below. First, we need the following well-known lemma.

\begin{lemma}\label{L:mu-on-proj}
The natural transformation 
\[
\mu_{A} : A \otimes \blank \lra (A^{\ast}, \blank)
\]
is an isomorphism whenever $A$ is a finitely generated projective.
\end{lemma}

\begin{proof}
 Let $M$ be an arbitrary left $\Lambda$-module. The map 
 $\mu_{A}(M) : A \otimes M \lra (A^{\ast}, M)$ is given by $\mu_{A}(M)(a \otimes m)(l) :=  l(a)m$ for any $a \in A$, $m \in M$, and $l \in A^{\ast}$. Assume now that $A$ is a finitely generated projective and let $\{e_{i}, f_{i}\}$ be a finite projective basis of $A$. It is easy to see that the map 
\[
(A^{\ast}, M) \lra A \otimes M : g \mapsto \sum e_{i} \otimes g(f_{i})
\]
where $g \in (A^{\ast}, M)$, is the inverse of $\mu_{A}(M)$.
\end{proof}

\begin{proposition}\label{P:fp implies s=t}
 If $A$ is finitely presented, then $\injt(A) = \mathfrak{t}(A)$.
\end{proposition}

\begin{proof}
 It suffices to show that, under the above assumption, the map 
\[
\mu_{A}(I) : A \otimes I \lra (A^{\ast}, I)
\]
from~\eqref{Eq:mu-square} is an isomorphism. Let $ P_{1} \to P_{0} \to A \to 0$
be a finite presentation of~$A$. Using the functoriality of $\mu$, we have a commutative diagram 
\[
\xymatrix
	{
	P_{1} \otimes I \ar[r] \ar[d]^{\mu_{P_{1}}(I)}  
	& P_{0} \otimes I \ar[r] \ar[d]^{\mu_{P_{0}}(I)}
	& A \otimes I \ar[r] \ar[d]^{\mu_{A}(I)} 
	& 0
\\
	(P_{1}^{\ast}, I) \ar[r]
	& (P_{0}^{\ast}, I) \ar[r]
	& (A^{\ast}, I) \ar[r]
	& 0
	}
\] 
Since the tensor product is right-exact, the top row is exact. Since $I$ is injective, the bottom row is exact, too. By Lemma~\ref{L:mu-on-proj}, the $\mu_{P_{i}}(I)$ are isomorphisms. It follows that $\mu_{A}(I)$ is an isomorphism, too.
\end{proof}

\begin{remark}\label{R:t}
 \textcolor{purple}{If $A$ is finitely presented, then $\mathfrak{t}(A) \simeq  \Ext^{1}_{\Lambda^{\mathrm{op}}}(\Tr A, \Lambda)$, as was first established in~\cite[Proposition 6.3]{A66}.} Here $\Tr A$ is the transpose of $A$.
\end{remark}

Next we want to show that the injective torsion functor commutes with filtered colimits and, \textcolor{purple}{therefore}, coproducts (see \cite[Corollary 10.5]{MR-1}). 
This is a consequence of the following well-known fact.

\begin{lemma}
Let $X,Y$ be left $\Lambda$-modules and $\blank\otimes X\overset{\alpha}{\longrightarrow} \blank\otimes Y$ be a natural transformation between the tensor \textcolor{purple}{product} functors on the category $\Mod$-$\Lambda$ of all right $\Lambda$-modules.  Then the functor $F:\Mod$-$\Lambda \to \ab$ defined by the exact sequence 
\[
0\longrightarrow F\longrightarrow\blank\otimes X\overset{\alpha}{\longrightarrow} \blank\otimes Y
\]
commutes with filtered colimits and coproducts.
\end{lemma}

\begin{proof}Let $A=\underset{\longrightarrow}{\lim}\ A_i$ be a left module represented as the filtered colimit of finitely presented modules.  Since both 
$\blank\otimes X$ and $\blank\otimes Y$ commute with coproducts and colimits, we have an exact sequence 
\[
\xymatrix
	{
	0 \ar[r]  
	& F(A) \ar[r] 
	& \varinjlim(A_i\otimes X) \ar[r]^{\alpha_A}
	& \varinjlim( A_i\otimes Y)
	}
\] 
Since the functor ``filtered $\underset{\longrightarrow}{\lim}$'' is exact in the category of abelian groups, there is also an exact sequence 
\[
\xymatrix
	{
	0 \ar[r]  
	& \varinjlim F(A_{i}) \ar[r] 
	& \varinjlim(A_i\otimes X) \ar[r]^{\alpha_A}
	& \varinjlim( A\otimes Y)
	}
\] 
It follows that we have a functorial isomorphism 
\[
\underset{\longrightarrow}{\lim }\ F(A_i) \simeq F(A)
\]
i.e., $F$ commutes with filtered colimits.  An entirely similar argument shows that $F(\coprod M_i) \simeq\coprod F(M_i)$ since taking coproducts is also exact in the category of abelian groups and the tensor product functors commute with coproducts.\footnote{\textcolor{purple}{This also follows from the fact that any direct sum is a filtered colimit of finite direct sums.}}  \end{proof}

\begin{proposition}\label{P:s-comm-colim}
For any left module $B$, the functor 
\[
\blank \ot B : \Mod\text{-}\Lambda \to \ab
\] 
commutes with filtered colimits and coproducts. In particular, the torsion functor 
$\injt = \blank \ot \Lambda$ commutes with filtered colimits and coproducts.
\end{proposition} 

\begin{proof}
Indeed, $\blank\ot B$ is determined by the exact sequence of functors 
\[
0\to \blank\ot B\to \blank\otimes B\to \blank\otimes I  
\] 
\end{proof}

If the ring is left semihereditary then the Bass torsion submodule of a finitely presented \texttt{right} module splits off \textcolor{purple}{(see, for example, \cite[Proposition~6]{M-10})}. Proposition~\ref{P:s-comm-colim} now yields a suitable extension of this result to arbitrary modules.

\begin{corollary}
 Suppose $\Lambda$ is \texttt{left} semihereditary and $A$ is a \texttt{right}  $\Lambda$-module. Then the inclusion $\injt(A) \subseteq A$ is pure.
\end{corollary}

\begin{proof}
 If $A$ is finitely presented, then $\injt(A)$ is a direct summand of $A$, as we just remarked. A general $A$ can be represented as a filtered colimit of finitely presented modules:
 $A \simeq \varinjlim A_{i}$ with structure maps $\varphi_{ij} : A_{i} 
 \lra A_{j}$. Applying the functor~$\injt$, we have a directed system $\injt(\varphi_{ij}) : \injt(A_{i})  \lra \injt(A_{j})$. As $\injt$ preserves filtered colimits, we have an isomorphism 
\[
\varinjlim (\injt(A_{i})) \simeq \injt(\varinjlim A_{i}) = \injt(A)
\]
In fact, it is easy to check that the inclusion $\injt(A) \to A$ is isomorphic to the colimit of the inclusions $\injt(A_{i}) \to A_{i}$. Since those inclusion are split monomorphisms, we have the desired claim.
 
 \end{proof}

Having brought colimits into the picture, we can now see what makes $\mathfrak{t}$ different from $\injt$: the Bass torsion does not, in general, preserve filtered colimits. To show this, we return to Example~\ref{E:extremes}. Represent $\Q$ as a filtered colimit of its finitely generated submodules: $\Q = \varinjlim A_{i}$. Then $\mathfrak{t}(A_{i})$  agrees with the classical torsion of $A_{i}$, which is zero. If $\mathfrak{t}$ preserved filtered colimits, we would have $\mathfrak{t}(\Q) = \{0\}$, a contradiction.

The next result makes the previous observation precise.

\begin{proposition}The torsion functor $\injt$ is the largest subfunctor of the Bass torsion functor $\mathfrak{t}$ that commutes with filtered colimits.
\end{proposition}

\begin{proof}
We have already seen in Proposition~\ref{P:fp implies s=t} that the inclusion $\injt \subseteq \mathfrak{t}$ evaluates to an equality on finitely presented modules. We have also seen in~Proposition~\ref{P:s-comm-colim}  that $\injt$ commutes with filtered colimits. Suppose that $\mu : \mathfrak{u} \to \mathfrak{t}$ is a subfunctor of $\mathfrak{t}$ commuting with filtered colimits.  Let $A=\underset{\longrightarrow}{\lim}\ A_i$ be a right module expressed as a filtered colimit of finitely presented modules.  The components 
$\mu_i : \mathfrak{u}(A_i)\to \mathfrak{t}(A_i) = \injt(A_i)$ induce a family of monomorphisms $\tilde{\mu}_i : \mathfrak{u}(A_i)\to \injt(A_i)$. By applying the exact functor $\underset{\longrightarrow}{\lim}$ to this family we have a monomorphism $\underset{\longrightarrow}{\lim}\ \tilde \mu_i \ :\underset{\longrightarrow}{\lim}\ \mathfrak{u}(A_i)\to \underset{\longrightarrow}{\lim}\ \injt(A_i)$.  Since $\mathfrak{u}$ and $\injt$ commute with filtered colimits, this monomorphism is equivalent to a monomorphism $\mathfrak{u} (A)\to \injt(A)$.  This is easily seen to be natural in $A$, establishing that $ \mathfrak{u} \subseteq \injt$.    
\end{proof}

One property of torsion shared by both the classical torsion over commutative domains and the Bass torsion over arbitrary rings, is that both functors are radicals, i.e., the quotient of any module modulo its torsion submodule has zero torsion. Now we show that $\injt$ has this property, too.

\begin{theorem}\label{T:residual-torsion}
$\injt$ is a radical, i.e., $\injt(A / \injt(A)) \simeq \{0\}$ for any module $A$.
\end{theorem}

\begin{proof}
As we just mentioned, for any $A$, $\mathfrak{t}(A / \mathfrak{t}(A)) = \{0\}$, where $\mathfrak{t}$ is the Bass torsion functor.  Since $\injt$ is a subfunctor of $\mathfrak{t}$ by Proposition~\ref{P:s-is-a-sub-of-t}, we have 
$\injt(A/ \mathfrak{t}(A)) = \{0\}$ for any $A$. Assuming now that is $A$ finitely presented we have, since $\injt(A) = \mathfrak{t}(A)$ by Proposition~\ref{P:fp implies s=t}, the desired assertion for finitely presented modules. 

If $A$ is now arbitrary, let $A =\underset{\longrightarrow}{\lim}\ A_i$ be a representation of $A$ as a filtered colimit of finitely presented modules. Then $\injt(A_i/\injt(A_i)) = \{0\}$ for each~$A_i$ and
\begin{align*}
	\injt(A/\injt(A))
	&  \simeq  \injt\bigg(\frac{\underset{\longrightarrow}{\lim}\ A_i} 					{\injt(\underset{\longrightarrow}{\lim}\ A_i)}\bigg)
	& 
\\
	& \simeq \injt\bigg(\frac{\underset{\longrightarrow}{\lim}\ A_i}						{\underset{\longrightarrow}{\lim}\ \injt(A_i)}\bigg)
	& (\text{since $\injt$ commutes with filtered colimits})
\\
	& \simeq \injt\bigg(\underset{\longrightarrow}{\lim}\ \frac{A_i}{\injt(A_i)}\bigg)
	& (\text{since filtered colimits preserve exactness})
\\
	&  \simeq \underset{\longrightarrow}{\lim}\, \injt\bigg(\frac{A_i}{\injt(A_i)}\bigg) 
	&  (\text{since $\injt$ commutes with filtered colimits})
\\
	& = \{0\}
\end{align*} 
\end{proof}

\begin{remark}\label{R:low-torsion}
\textcolor{purple}{In general, the radical $\injt$ is not idempotent, which can be seen as follows. By Proposition~\ref{P:fp implies s=t}, $\injt(A) =\mathfrak{t}(A)$ for any finitely presented $A$. It was shown in~\cite[Proposition~17]{M-10} that, over a commutative artinian local ring, $\mathfrak{t}(A)$ is contained in the radical of $A$ for \texttt{any} finitely generated nonzero $A$, making $\mathfrak{t}$ nilpotent on $A$. We would be done if we can now find a finitely generated $A$ with a nonzero $\mathfrak{t}(A)$. Let~$\Bbbk$ be a field and $\Lambda := \Bbbk[x,y]/(x,y)^{2}$. This is a finite-dimensional local algebra whose socle is minimally generated by the classes of $x$ and $y$. Hence~$\Lambda$ is indecomposable and not self-injective. Then there is an almost split sequence beginning with~$\Lambda$. It corresponds to a nonzero element of 
$\Ext^{1}(\Tr D \Lambda, \Lambda)$, the latter being isomorphic, by 
Remark~\ref{R:t}, to $\mathfrak{t}(D\Lambda)$ (here $D$ is the duality with respect to $\Bbbk$). Thus we can take $A := D \Lambda$.}
\end{remark}

It is helpful to update our terminology and introduce the familiar classical notions of torsion module and torsion-free module into the new context. As before, let $\iota : \prescript{}{\Lambda}{\Lambda} \to I$ be the injective envelope.

\begin{definition}\label{D:torsion-mod}
 A right module $A$ is said to be a \texttt{torsion module} if the canonical monomorphism $ \injt(A) \to A$ is an isomorphism or, equivalently, the map
\[
1 \otimes \iota : A \otimes \Lambda \to A \otimes I
\]
is zero. 
\end{definition}

\begin{proposition}\label{P:t-quot-coprod}
The class of torsion modules is closed under coproducts and quotient modules.
\end{proposition}

\begin{proof}
\textcolor{purple}{This result holds in greater generality: the torsion class of any preradical is closed under quotient modules and coproducts, see~\cite[VI, Proposition 1.2]{S-75}.}
\end{proof}

\begin{remark}\label{R:t-not-tt}
\textcolor{purple}{In general, the class of torsion modules is not closed under extensions. Otherwise this class would be the torsion class of a torsion theory in the sense of Dickson~\cite[VI, Proposition~2.1]{S-75}. The latter is equivalent, by~\cite[VI, Proposition~2.3]{S-75}, to saying that  $\injt$ is an idempotent radical, which is not always the case by Remark~\ref{R:low-torsion}.}
\end{remark}

\begin{definition}\label{D:torsionfree-mod}
 A \textcolor{purple}{right} module $A$ is said to be \texttt{torsion-free} if $ \injt(A) = \{0\}$ or, equivalently, the map  $1 \otimes \iota : A \otimes \Lambda \to A \otimes I$ is monic. 
\end{definition}

The short exact sequence of endofunctors
\begin{equation}\label{Eq:s-inverse}
 0 \lra \injt \lra \mathbf{1} \lra \mathbf{1}/\injt \lra 0
\end{equation}
on the category of all right $\Lambda$-modules makes it reasonable to set $\injt^{-1} := \mathbf{1}/\injt$. Theorem~\ref{T:residual-torsion} justifies the following definition.

\begin{definition}
The quotient module $\injt^{-1}(A)$ will be called the torsion-free quotient module of $A$.
\end{definition}

\begin{proposition}\label{P:tf-sub-prod}
\textcolor{purple}{The class of torsion-free modules is closed under submodules and products.}
\end{proposition}

\begin{proof}
\textcolor{purple}{This result holds in greater generality: the torsion-free class of any preradical is closed under submodules and products, see~\cite[VI, Proposition 1.2]{S-75}.}
\end{proof}

It is a general property of radical functors that the corresponding torsion-free class\footnote{Strictly speaking, this should be called the \texttt{pretorsion-free} class.} is a reflective subcategory. In our case, this means that $\injt^{-1}$, called a reflector and viewed as a functor from the category of all modules to the category of torsion-free modules (but not all modules!) is left adjoint to the inclusion functor. Thus, we have 

\begin{proposition}\label{P:reflector-colim}
The functor $\injt^{-1}$ from the category of all \textcolor{purple}{right} modules to the category of torsion-free \textcolor{purple}{right} modules preserves all colimits. In particular, the class of torsion-free modules has all colimits.  \qed
\end{proposition}

\begin{remark}
It is tempting to use the short exact sequence~\eqref{Eq:s-inverse} to show that~$\injt$ preserves all colimits, too. However such an argument would be incorrect since, in that sequence, $\injt^{-1}$ denotes a functor from the category of all modules to itself, whereas the codomain of $\injt^{-1}$ in the proposition is the subcategory determined by the torsion-free modules. In fact, as we shall see in Proposition~\ref{P:s-epi} below, $\injt$ preserves all colimits if and only if it is the zero functor. It is a general fact that colimits in a reflective subcategory can be computed as images under the reflector (in our case, $\injt^{-1}$) of colimits in the ambient category.
\end{remark}

The next three results deal with the vanishing and the exactness properties of the injective torsion functor. Since tensoring with a flat module is an exact functor, we have

\begin{lemma}\label{L:flat=0}
If $A$ is flat, then $\injt(A) = \{0\}$, i.e., each flat module is torsion-free. \qed
\end{lemma}

\textcolor{purple}{In view of Corollary~\ref{C:s-pres-mono}, we have}

\begin{corollary}
 \textcolor{purple}{If $A$ is flat, then $\injt(B) = 0$ for any submodule $B$ of $A$}. \qed
\end{corollary}

\begin{remark}
\textcolor{purple}{Specializing, in the preceding corollary, to the case of a finitely presented submodule $B$, we have, in view of Proposition~\ref{P:fp implies s=t}, 
that a finitely presented submodule of a flat module is torsionless (i.e., its Bass torsion submodule is zero).}
\end{remark}

\begin{proposition}\label{P:s-epi}
The following conditions are equivalent \textcolor{blue}{for the torsion radical~$\injt$ on right 
$\Lambda$-modules}:

\begin{itemize}
 \item[a)] $\injt$ preserves epimorphisms;
 
 \item[b)] $\injt$ is the zero functor;
 
 \item [(c)] \textcolor{purple}{$\injt$ vanishes on finitely presented modules};
 
 \item[d)] \textcolor{purple}{$\injt$ vanishes on injectives};
 
 \item[e)] $\prescript{}{\Lambda}{\Lambda}$ is absolutely pure;
 
 \item[f)] $\Lambda$ is \texttt{left} $FP$-injective, i.e., $\Ext^{1}_{\Lambda}(M, \prescript{}{\Lambda}{\Lambda}) = \{0\}$ for all finitely presented left  $\Lambda$-modules $M$.
\end{itemize}

In particular, if $\Lambda$ is selfinjective on the left, then $\injt$ is the zero functor \textcolor{purple}{(on right modules)}.

\textcolor{purple}{If $\Lambda$ is two-sided noetherian, then the above conditions are equivalent to $\Lambda$ being a \texttt{right} IF ring (i.e., all injective right $\Lambda$-modules being flat)}.
\end{proposition}

\begin{proof}
 Given any \textcolor{purple}{right module} $A$, choose an epimorphism $P \to A \to 0$ with $P$ projective. Assuming that $\injt$ preserves epimorphisms, we have an epimorphism $\injt(P) \to \injt(A) \to 0$. By Lemma~\ref{L:flat=0}, $\injt(P) = \{0\}$ and, therefore, 
 $\injt(A) = \{0\}$, thus proving the implication a) $\Rightarrow$ b). The converse is trivial. 

\textcolor{purple}
{To show the equivalence of b) and c), pick any right module and represent it as a filtered colimit of finitely presented modules. The result now follows from the fact that $\injt$ preserves filtered colimits (Proposition~\ref{P:s-comm-colim}).}

\textcolor{purple}{Since any module can be embedded in an injective module, Corollary~\ref{C:s-pres-mono} shows the equivalence of b) and d).}
 
\textcolor{purple}{To prove the remaining equivalences, notice that~$\injt$ is the zero functor if and only if $A \otimes \Lambda \to A \otimes I$ is monic for any $A$. In terminology of~\cite{Mad}, this means that $\prescript{}{\Lambda}{\Lambda}$ is absolutely pure, which is equivalent to the condition $\Ext^{1}_{\Lambda}(M, \prescript{}{\Lambda}{\Lambda}) = \{0\}$ for all finitely presented left $\Lambda$-modules $M$~\cite[Proposition 1]{Meg}.}

\textcolor{purple}{Finally, assume that $\Lambda$ is two-sided noetherian. If all injective right $\Lambda$-modules are flat, then d) follows immediately (even without the noetherian assumption). Conversely, suppose the above conditions hold and let $J$ be an injective right $\Lambda$-module. The fact that $J$ is flat then follows from a result of H.~Sato~\cite[Lemma 1.4]{Sa}. For the convenience of the reader, we recall his argument, which relies on the theorem of Lazard~\cite[Th\'{e}or\`{e}me 1.2]{L} characterizing flat modules. Thus let $M$ be any finitely presented right module and $f : M \lra J$ an arbitrary homomorphism. Since $\Lambda$ is noetherian, $f(M)$, being finitely generated, is finitely presented and therefore 
$\mathfrak{t}(f(M)) = \injt(f(M)) = 0$, where the last equality holds by the assumption. Choose an epimorphism $g : P \lra f(M)$ with $P$ finitely generated projective. Dualizing into the ring, we have that $f(M)^{\ast}$ is a submodule of the finitely generated projective $P^{\ast}$, and, by the noetherian assumption is also finitely generated. Choose an epimorphism $h : Q \lra f(M)^{\ast}$ with $Q$ finitely generated projective. Dualizing again, we have that $f(M)^{\ast\ast}$ is a submodule of the finitely generated 
projective~$Q^{\ast}$. Since $f(M)$ embeds in $f(M)^{\ast\ast}$ we have that it also embeds in $Q^{\ast}$. By the injectivity of $J$, the embedding of $f(M)$ in $J$ can be extended to $Q^{\ast}$, which shows that $f$ factors through a finitely generated projective. By the theorem of Lazard,~$J$ is flat.}
\end{proof}

\textcolor{purple}{The just proved result allows to prove the nonexistence of a functorial splitting for the injective torsion.} 

\begin{corollary}
 \textcolor{purple}{Suppose the injective torsion splits off as a functor, i.e., the exact sequence $0 \lra \injt \lra \mathbf{1} \lra \injt^{-1} \lra 0$ splits. Then $\injt =0$, i.e., $\Lambda$ is left absolutely pure.}
 \end{corollary}

\begin{proof}
 \textcolor{purple}{Under the assumption, $\injt$ becomes a quotient functor of the identity functor~$\mathbf{1}$ and hence preserves epimorphisms. The result follows.}
\end{proof}

\begin{remark}
\textcolor{blue}{The foregoing corollary and the proof of the implication a) $\Rightarrow$ b) of Proposition~\ref{P:s-epi} show that the existence of a functorial splitting of $\injt$ restricted to a full subcategory of right $\Lambda$-module having enough projectives still implies the vanishing of $\injt$ restricted to that category.}
 \end{remark}

\begin{proposition}\label{P:s-of-pex-is-ex}
If $0 \to A' \to A \to A'' \to 0$ is a pure exact sequence, then the \textcolor{purple}{induced} sequence 
 \[
 0 \lra \injt(A') \lra \injt(A) \lra \injt(A'') \to 0
 \]
 is exact.
\end{proposition}

\begin{proof}
\textcolor{purple}{A pure exact sequence is a directed colimit of split exact sequences. Since~$\injt$ preserves filtered colimits, the sequence in question is a filtered colimit of similar sequences arising from split exact sequences. Each such sequence is exact 
because~$\injt$ is an additive functor. The result now follows since filtered colimits preserve exactness.}
\smallskip
\newline\textit{Second Proof.} Tensor the cosyzygy sequence $0 \to \Lambda \to I \to \Sigma \Lambda \to 0$ with the given pure exact sequence and use the snake lemma.\footnote{\textcolor{purple}{The first proof was suggested by the referee. We also keep the original proof because it applies to abelian categories without \textbf{AB5}.}}
\end{proof}

\textcolor{purple}{The next observation was suggested by the referee.}

\begin{corollary}
 \textcolor{purple}{Both the torsion-free and the torsion class of $\injt$ are closed under pure extensions.} \qed
\end{corollary}

Now we want to examine the properties of the torsion functor 
$\injt$ when the injective envelope of $\prescript{}{\Lambda}{\Lambda}$ is flat. 
 
\begin{proposition}\label{P:inj-torsion}
Suppose the injective envelope of $\prescript{}{\Lambda}{\Lambda}$ is flat. Then:
\begin{enumerate}

\item $\injt \simeq \Tor_{1}(\blank, \Sigma \Lambda)$.
\smallskip

\item  If $0 \to A' \to A \to A'' \to 0$ is a short exact sequence of right 
$\Lambda$-modules, then \textcolor{purple}{there is a homomorphism $\delta : \injt(A'') \to A' \otimes \Sigma \Lambda$ such that} the induced sequence 
 \[
 0 \to \injt(A')  \to  \injt(A) \to  \injt(A'') \overset{\textcolor{purple}{\delta}}\to
 A' \otimes \Sigma \Lambda \to A \otimes \Sigma \Lambda \to A'' \otimes \Sigma \Lambda \to 0
 \]
is exact. In particular, $\injt$ is left-exact.
\smallskip 

\item The natural transformation $\injt(\injt \to 1) = \injt^{2} \to \injt$ is an isomorphism, i.e., \textcolor{purple}{$\injt$ is an idempotent radical and thus the torsion class of $\injt$ is the torsion class of a torsion theory.} \textcolor{purple}{That torsion theory is hereditary, i.e., its torsion class is closed under submodules.}
\smallskip

\item The torsion-free class of $\injt$ is closed under extensions.
\smallskip
\item \textcolor{purple}{The torsion class of $\injt$ is closed under extensions}.
\end{enumerate}
\end{proposition}

\begin{proof}

(1) Tensor the short exact sequence 
\[
0 \lra \Lambda \lra I \lra \Sigma \Lambda \lra 0
\]
with an arbitrary right module and pass to the long exact sequence. 
\smallskip
 
(2) Follows from \textcolor{purple}{the definition of $\injt$ and} the snake lemma; for \textcolor{purple}{$\delta$ take the connecting homomorphism.}
\smallskip
 
 (3) By (2), $\injt$ is left-exact. Thus the short exact sequence 
\[
0 \lra \injt(A) \lra A \lra \injt^{-1}(A) \lra 0
\] 
gives rise to the exact sequence 
\[
0 \lra \injt^2(A) \lra \injt(A) \lra \injt(\injt^{-1}(A)),
\] 
where $\injt(\injt^{-1}(A)) = \{0\}$ by Theorem~\ref{T:residual-torsion}, \textcolor{purple}{showing that $\injt$ is idempotent. By~\cite[VI, Proposition~2.3]{S-75}, idempotent radicals bijectively correspond to torsion theories, which proves the first claim.  The second claim follows from (2) and the bijective correspondence~\cite[VI, Proposition 3.1]{S-75} between hereditary torsion theories and left-exact radicals.}
\smallskip

(4) Apply $\injt$ to a short exact sequence $ 0\to A'\to A\to A''\to 0$. The resulting exact sequence $0\to  \injt(A')\to \injt(A)\to \injt(A'')$ shows that if $\injt(A')$ and $\injt(A'')$ vanish, then so does $\injt(A)$, which means that the torsion-free class is closed under extensions.

(5) \textcolor{purple}{As we already mentioned, the torsion class of $\injt$ is the torsion class of a torsion theory. Such classes are closed under extensions~(\cite[VI, Proposition 2.1]{S-75}). }
\end{proof}

\begin{remark}
\textcolor{blue}{Several types of rings are known to have a flat injective envelope. 
The field of quotients of a commutative domain is its injective envelope. Tensoring with the field of quotients amounts to the localization at the zero ideal, and since localization is an exact functor, the field of quotients is flat. More generally, the injective envelope of a two-sided Ore domain (on either side) is flat; this follows from~\cite[Exercise~10.20]{Lam}. Commutative Gorenstein rings also have this property~\cite[Proposition~5.1.2 and Lemma~5.1.1]{Xu}. For various characterizations of commutative noetherian rings with a flat envelope, see~\cite{KS} and~\cite{K-A}. There are also several obvious choices, for example the (right) IF rings, i.e., the rings all of whose (right) injectives are flat. Such rings were mentioned in Proposition~\ref{P:s-epi}, which shows that the only possibly interesting cases are the ones with non-noetheiran rings. See~\cite[Theorems 6.8 - 6.11]{Faith} for characterizations of IF rings. Another obvious choice is the class of all rings whose injective envelope is projective. The (left) hereditary rings with this property have been classified~\cite[Theorem 3.2]{CR}; based on that result, it was later shown that those are precisely the left hereditary rings whose category of left modules modulo projectives is abelian~\cite[Theorem 9.5]{MZ}.  }
\end{remark}

\renewcommand{\mod}{\mathrm{mod}}

Now we want to give another description of the torsion submodule, this time using the notion of colimit extension. \textcolor{purple}{Recall that any additive functor $F$ on the category of finitely presented right $\Lambda$-modules extends uniquely, up to isomorphism, to a functor $\overset{\to}{F}$, defined on the category of all modules, which preserves filtered colimits. $\overset{\to}{F}$ will be called here the colimit extension of $F$; see~\cite[Section 10]{MR-1} for more details.}
\smallskip

Let~$\mathscr{F}$ denote the class of flat right $\Lambda$-modules, and $Rej(A, \mathscr{F})$ -- the reject of~$\mathscr{F}$ in the right module $A$; \textcolor{purple}{see~\cite[pp. 109-112]{AF-92} for details on the reject}. The restriction of $Rej(\blank, \mathscr{F})$ to the full subcategory  determined by the finitely presented modules will be denoted by $rej(\blank, \mathscr{F})$.

\begin{proposition}\label{P:s=colim-ext}
 $\injt \simeq \overset{\lra}{rej}(\blank, \mathscr{F})$, i.e., the torsion functor is isomorphic to the colimit extension of the reject of flats restricted to finitely presented modules.
\end{proposition}

\begin{proof}
 It is easy to see that the Bass torsion functor $\mathfrak{t}$ is just the reject of  $\Lambda$. As~$\Lambda$ is flat, ${Rej}(\blank, \mathscr{F})$ is a subfunctor of $\mathfrak{t}$. On the other hand, let $F$ be a flat module and $f : A \to F$ an arbitrary homomorphism. Since torsion is a subfunctor of the identity functor, we have a commutative diagram 
 \[
\xymatrix
	{
	\injt(A) \ar[r] \ar[d]_{\injt(f)}
	& A\ar[d]^{f}
\\
	\injt(F)  \ar[r]
	& F
	}
\] 
By Lemma~\ref{L:flat=0}, $\injt(F) = \{0\}$, and therefore torsion is a subfunctor of the reject. Thus we have a chain of functors
\[
\injt \subseteq {Rej}(\blank, \mathscr{F}) \subseteq \mathfrak{t}
\]
Restricting it to finitely presented modules we have an identification of the end terms. Therefore ${rej}(\blank, \mathscr{F})$ coincides with $\injt$ restricted to finitely presented modules. Passing to the colimit extensions and using the fact that $\injt$ preserves filtered colimits, we have the desired result.
\end{proof}

This alternative description of the torsion functor~$\injt$, together with the notion of right cosatellite~\cite[Definition 5.7]{MR-1} and~\cite[Proposition 5.9]{MR-1}, show that 
\[
\injt^{-1}(A) \simeq C^{1}(A \otimes \blank)(\Lambda),
\]
which allows us to rewrite the short exact sequence 
\[
0 \lra \injt(A) \lra A \lra \injt^{-1}(A) \lra 0
\]
in the following form
\begin{equation}\label{Eq:Rej}
 0 \lra \overset{\lra}{rej}(A, \mathscr{F}) \lra A \lra C^{1}(A \otimes \blank)(\Lambda)
\lra 0.
\end{equation}
In summary,
\begin{equation}\label{Eq:Rej-mult}
 \injt(A) = \overset{\lra}{rej}(A, \mathscr{F}) \quad \text{and} \quad
\injt^{-1}(A) = C^{1}(A \otimes \blank)(\Lambda)
\end{equation}
\smallskip

\begin{corollary}
The right cosatellite $C^{1}(A \otimes \blank)(\Lambda)$, viewed as a functor of $A$ from all modules to torsion-free modules, preserves all colimits.
\end{corollary}

\begin{proof}
 Follows from Proposition~\ref{P:reflector-colim}.
\end{proof}

We end the discussion of torsion by looking at its iteration. This yields a descending chain 
\begin{equation}\label{chain}
 A \supset \injt(A) \supset \ldots \supset \injt^{n}(A) \supset \ldots
\end{equation}
of submodules of $A$. By~\cite[Lemma~4.1]{MR-1},
these submodules are determined uniquely not only as modules up to isomorphism but also as submodules. 
It would be of interest to investigate this chain and, in particular, its asymptotic properties. We only know its behavior in some simple cases. Recapping Lemma~\ref{L:flat=0}, Proposition~\ref{P:s-epi}, Proposition~\ref{P:inj-torsion}, and Corollary~\ref{C:s-pres-mono}, we have:

\begin{itemize}
 \item If $A$ is flat, then $\injt(A) = \{0\}$.
 \smallskip
 
 \item If $\Lambda$ is selfinjective, then $\injt(A) = \{0\}$ for all $A$.
  \smallskip
  
  \item If the injective envelope of $\Lambda$ is flat, then $\injt^{2} = \injt$ and the chain stabilizes at the first step. This generalizes the same property for classical torsion over commutative domains. 
  \smallskip
 
 \item The injective torsion of any submodule of a flat module is zero. In particular, this holds for syzygy modules in projective resolutions.
\end{itemize}

It is easy to construct modules $A$ with $\injt(A) = A$. This happens, for example, when $A \otimes I = \{0\}$, where $I$ is the injective envelope of 
$\Lambda$. Over an arbitrary ring, all nonzero finitely presented modules which coincide with their injective torsion (i.e., with their Bass torsion) can be characterized: each such module is the transpose of a module of projective dimension one~\cite[Proposition 5]{M-10}.

Finally, assume that $\Lambda$ is a commutative artinian local ring and $A$ is an arbitrary finitely generated $\Lambda$-module. It was shown 
in~\cite[Proposition 17]{M-10} that $\mathfrak{t}(A)$ does not contain minimal generators of $A$. Since in this case  $\injt(A) = \mathfrak{t}(A)$, we have that: a) $\injt(A)$ is a submodule of \textcolor{purple}{the radical of} $A$, b) the chain~\eqref{chain} stabilizes at $\{0\}$, and c) the length of the chain does not exceed the Loewy length of $A$.

\section{Cotorsion over arbitrary rings}\label{S:cotorsion}

Given the generality of our definition of torsion, one may ask if that definition can be formally dualized to yield a definition of cotorsion. The answer is yes and, as we shall see, this can be done with little effort. However, compared with torsion, one faces the surprising fact that while there have been several competing definitions of cotorsion modules, there seems to be no definition of the \texttt{cotorsion module of a module}. Our next goal is to provide such a definition for arbitrary modules over arbitrary rings. \textcolor{purple}{Our plan is to introduce a certain quotient functor of the identity functor, i.e., a precoradical, and then define the cotorsion module of a module as the image of the quotient transformation computed at the module. We shall also show that our precoradical is in fact a coradical.} 

Recall that the injective torsion submodule $\injt(A) = (A \ot \blank)(\Lambda)$ of a right $\Lambda$-module $A$ was defined as the kernel of the map $A \otimes \Lambda \to A \otimes I$, where $I$ is an injective container of $\prescript{}{\Lambda}{\Lambda}$. The fact that $\injt(A)$ is a submodule of $A$ is due in part to the canonical isomorphism $A \otimes \Lambda  \cong A $. In a dual approach, it could be expected that one should use the other well-known canonical isomorphism, 
$\Hom(\Lambda, C)  \cong C$, where $C$ is an arbitrary (say, left) $\Lambda$-module. This leads to the contravariant functor $\Hom(\blank, C)$. Since the torsion \textcolor{purple}{submodule of $A$} was defined as \textcolor{purple}{the value of} the \texttt{injective} stabilization of the functor $A \otimes \blank$ \textcolor{purple}{at $\Lambda$}, in a dual approach one should be looking at the \texttt{projective} stabilization of  the contravariant functor $\Hom(\blank, C)$ \textcolor{purple}{evaluated at $\Lambda$}.\footnote{\textcolor{purple}{The hypothetical choice between projective or injective stabilization is imaginary: since $\Hom$ is left-exact in both variables, the injective stabilization of any of them is zero.}} Viewing $\Hom(\blank, C)$ as a \texttt{covariant} functor on the opposite category (which is never a module category but is still abelian) one is led to consider the contravariant Hom functor modulo injectives, customarily denoted by $\oh$.\footnote{Warning: we repeat that this is not the injective stabilization of \textcolor{purple}{any of the two Hom functors}, which are zero. We are forced to use this confusing notation since, historically, overline denotes the injective stabilization of a functor, but in the case of the Hom functors it \textcolor{purple}{also} denotes Hom modulo injectives. \textcolor{purple}{See a more detailed discussion of this below, between Theorems~\ref{T:DA} and~\ref{T:cotorthm}}.}
 This motivates

\begin{definition}\label{D:cotorsion}
 Let $C$ be a left $\Lambda$-module. The cotorsion module of~$C$ is defined as $\injc(C) : = \underline{\Hom(\blank, C)}(\prescript{}{\Lambda}{\Lambda}) = \oh(\prescript{}{\Lambda}{\Lambda}, C)$.\footnote{The term in the middle denotes the projective stabilization of the \texttt{contravariant} functor $\Hom(\blank, C)$.} 
\end{definition}

Dropping the variable $C$ and switching to the categorical notation (i.e., dropping the symbol $\Hom\!$), we have 
$\injc = (\overline{\Lambda, \blank})$. In the spirit of Definition~\ref{D:tensor-active-vs-inert} we can now say that the cotorsion functor is the inert projective stabilization of the contravariant  functor $\Hom\!$.

In more detail, if 
\begin{equation}\label{cosyz-seq}
0 \lra \Lambda \overset{\iota}\lra I \lra \Sigma \Lambda \lra 0
\end{equation}
is an injective cosyzygy sequence, then the cotorsion module of $C$ is defined by the exact sequence

\begin{equation}\label{cotorsion-def}
 0 \lra (\Sigma\Lambda, C) \lra (I, C) \overset{(\iota, C)}\lra (\Lambda, C) \lra (\overline{\Lambda, C}) \lra 0.
\end{equation}
\textcolor{purple}{To justify the notation for the last term in this sequence, we first remark that any map $\Lambda \to C$ in the image of $(\iota, C)$  factors through the injective $I$ via $\iota$. On the other hand, if $\Lambda \to C$ factors through an injective, then, by~\eqref{cosyz-seq} and the defining property of injective modules,  this map extends along $\iota$. Thus the image of $(\iota, C)$ consists of all maps that factor through injectives, justifying the notation for the last term. }

The reader should verify that the contravariant Hom modulo injectives fits a pattern dual to that of the injective stabilization. In fact, for any additive functor~$F$, \textcolor{purple}{covariant or contravariant}, from modules to abelian groups, the projective stabilization of $F$ is defined as the cokernel of the natural transformation $L_{0}F \to F$, \textcolor{purple}{where $L_{0}F$ is the zeroth left derived functor of $F$ (for more details, see the discussion after Corollary~3.15 of~\cite{MR-1})}. If $F$ is a contravariant Hom functor, then each component of its projective stabilization comprises the classes of homomorphisms modulo the ones factoring through \texttt{injectives}. As a consequence, we have 

\begin{proposition}
 $\injc$ is a quotient functor of the identity functor on the category 
 $\Lambda$-$\Mod$ of left $\Lambda$-modules,  
 \textcolor{purple}{and therefore can be viewed as a quotient functor of the forgetful functor on $\Lambda$-$\Mod$ to abelian groups.}
\qed
\end{proposition}

\begin{corollary}\label{C:q-epi}
 $\injc$ preserves epimorphisms. In particular, if $\injc(C) = \{0\}$ and $D$ is a quotient module of $C$, then $\injc(D) = \{0\}$. \qed
\end{corollary}

\begin{corollary}
 \textcolor{purple}{The projective stabilization of $\injc$ is zero: 
 $\underline{\injc} = 0$.}
\end{corollary}

\begin{proof}
\textcolor{purple}{Similar to that of Corollary~\ref{C:inj-stab-s=0}: the projective stabilization of a functor is zero if and only if it preserves epimorphisms.}
\end{proof}

\begin{definition}
 The quotient module $\injc(C)$ will be called the cotorsion quotient module of $C$.
\end{definition}

\begin{example}
\textcolor{purple}{If $\Lambda := \Z$ and $A$ is an abelian group, then $\injc(A)$ is just the reduced group of $A$.}
\end{example}

\begin{example}
\textcolor{purple}{Trying to come up with a companion result for Corollary~\ref{C:Lambda-power}, we look at $\injc(\Lambda)$. This object varies from the zero module for selfinjective rings to the entire 
 $\Lambda$ for PIDs which are not fields (as they lack divisible elements). A more definitive answer can be given for the character module $\Lambda^{+} :=\, _{\Z}\!(\Lambda, \Q/\Z)$ of $\Lambda$. In fact, as we shall show in Corollary~\ref{C:exchange-formula}, that 
 $\injc((\Lambda^{N})^{+}) \simeq \injt(\Lambda^{N})^{+} \simeq 0$ for any indexing set $N$, where the latter isomorphism is provided by Corollary~\ref{C:Lambda-power}.}
\end{example}

\textcolor{purple}{A general property of $\injc(\Lambda)$ is given by the next observation.}

\begin{lemma}\label{L:q-L-bimodule}
 \textcolor{purple}{For any ring $\Lambda$, $\injc(\prescript{}{\Lambda}{\Lambda})$ is a $\Lambda$-$\Lambda$ bimodule, and the canonical map $\Lambda_{\Lambda} \cong (\prescript{}{\Lambda}{\Lambda}, \prescript{}{\Lambda}{\Lambda}) \lra \injc(\prescript{}{\Lambda}{\Lambda})$ is a bimodule homomorphism.}
\end{lemma}

\begin{proof}
\textcolor{purple}{By the foregoing discussion, $\injc(\prescript{}{\Lambda}{\Lambda}) = (\prescript{}{\Lambda}{\Lambda}, \prescript{}{\Lambda}{\Lambda})/I(\prescript{}{\Lambda}{\Lambda}, \prescript{}{\Lambda}{\Lambda})$, where the denominator denotes the endomorphisms of 
 $\prescript{}{\Lambda}{\Lambda}$ factoring through injectives. The right 
 $\Lambda$-module structures on the source and the target make the numerator a $\Lambda$-$\Lambda$ bimodule. Since, additively, the denominator is a subgroup of the numerator, we only need to check that the denominator withstands scaling on either side. Let 
\[
\xymatrix
	{
	\prescript{}{\Lambda}{\Lambda} \ar[rr]^{f} \ar[rd]_{g}
	&
	& \prescript{}{\Lambda}{\Lambda}
\\
	& J \ar[ru]_{h}
	}
\]
be a factorization with an injective $J$. Switching to the convention ``maps on the opposite side from scalars'', we have $f = g \circ h$. If 
$\alpha \in \Lambda$, then the right $\Lambda$-module structure on the domain of $g$ yields a factorization $\alpha f = \alpha g \circ h$. Similarly, the right $\Lambda$-module structure on the codomain of $h$ yields a factorization  $f \alpha = g \circ h\alpha$. The straightforward verifications are left to the reader.}
\end{proof}

The defining sequence~\eqref{cotorsion-def} gives rise to a short exact sequence of endofunctors on the category of (left) $\Lambda$-modules
\[
0 \lra \injc^{-1} \lra \mathbf{1} \lra \injc \lra 0
\]
which we view as the defining sequence for the endofunctor 
$\injc^{-1}$. 
We can rewrite it as a short exact sequence 
\begin{equation}\label{Eq:cotorsion-def}
 \xymatrix
	{
	0 \ar[r]
	& I(\Lambda, C) \ar[r]
	& (\Lambda, C) \ar[r]
	& (\overline{\Lambda, C}) \ar[r]
	& 0
	}
\end{equation}
of left $\Lambda$-modules, where $I(\Lambda, C)$ denotes the submodule of $(\Lambda, C)$ consisting of the maps 
$\Lambda \to C$ factoring through injectives.\footnote{\textcolor{purple}{The reader should verify that this is indeed a submodule and not just an abelian subgroup.}} \textcolor{purple}{This observation indicates the relevance of the trace of the class of injective modules. Before proceeding, the reader may benefit from reviewing the general properties of the trace, as expounded in~\cite[pp. 109-112]{AF-92}, and also recall, for a balanced point of view, the relevance of the reject of flats to the notion of torsion, as explained in Proposition~\ref{P:s=colim-ext}.}

\begin{lemma}\label{L:trace} 
Under the canonical isomorphism  $(\Lambda, C) \cong C : f \mapsto f(1)$, we have $I(\Lambda, C) \cong Tr(\mathscr{I}, C)$, i.e.,  
$I(\Lambda, C)$ identifies with  $Tr(\mathscr{I}, C)$, the trace in $C$ of the class~$\mathscr{I}$ of injective  $\Lambda$-modules.
\end{lemma}

\begin{proof}
Since any $k \in I(\Lambda, C)$ factors through an injective, $k(1)$ belongs to the trace of the injectives. 
Therefore $I(\Lambda, C)$ is contained in the trace. To show the reverse inclusion, suppose  that $l \in C$ can be written as $h(m)$ for some map $h : J \to C$, where $J$ is injective and $m \in J$. Define a map $g : \Lambda \to J$ by setting $g(1):= m$. Then $l = hg(1)$ and therefore $l \in I(\Lambda, C)$. Such elements $l$ generate the trace. Since $I(\Lambda, C)$ is a $\Lambda$-module, the trace is contained in $I(\Lambda, C)$.
\end{proof}


\textcolor{purple}{The next proposition refers to the trace of injectives viewed as an object in the category of finitely presented functors and natural transformations between them; see the beginning of 
Section~\ref{S:duality} and Theorem~\ref{T:afpthm} for the relevant definitions and results.}

\begin{proposition}
\textcolor{purple}{The covariant functor $Tr(\mathscr{I}, \blank)$ is finitely presented of projective dimension at most 1.}
Moreover, $\mathrm{proj\, dim}\, Tr(\mathscr{I}, \blank) = 0$ (i.e., $Tr(\mathscr{I}, \blank)$ is representable) if and only if $\Lambda$ is left selfinjective.
\end{proposition}

\begin{proof}
 The exact sequences~\eqref{cotorsion-def} yields a length-one projective resolution 
 \[
 0 \lra (\Sigma\Lambda, \blank) \lra (I, \blank) \lra Tr(\mathscr{I}, \blank) \lra 0
 \]
 of the trace, proving the first assertion. This sequence and Yoneda's lemma, show that the trace is projective if and only if \textcolor{purple}{$I \to \Sigma \Lambda$ is a split epimorphism if and only if $\Lambda \to I$ is a split monomorphism, which proves the second claim.}
\end{proof}

\begin{corollary}
 $Tr(\mathscr{I}, \blank)$, viewed as an endofunctor on the category of all left $\Lambda$-modules, preserves limits if and only if 
 $\Lambda$ is selfinjective. \textcolor{purple}{In that case, $Tr(\mathscr{I}, \blank)$ is the identity functor.}
\end{corollary}

\begin{proof}
 By a theorem of Eilenberg and Watts, a covariant functor preserves limits if and only if it is representable. By the above proposition, this is equivalent to 
 $\Lambda$ being selfinjective.
\end{proof}

\begin{remark}\label{R:pdq}
\textcolor{purple}{Since the projective resolution 
\[
0 \lra (\Sigma\Lambda, \blank) \lra (I,\blank) \lra 
(\prescript{}{\Lambda}{\Lambda}, \blank) \lra \injc \lra 0
\] 
of $\injc$ comes from the short exact sequence 
$0 \lra \Sigma\Lambda \lra I \lra \prescript{}{\Lambda}{\Lambda} \lra 0$ the defect
(see~\cite[Proposition 6.1]{MR-1}) of $\injc$ is zero.}
\end{remark}

The next result shows that over noetherian hereditary rings, the cotorsion module of a module is a direct factor of the module. More precisely, we have

\begin{proposition}\label{P:cotorsion-direct}
If $\Lambda$ is left noetherian and left hereditary and $C$ is a left $\Lambda$-module, then the epimorphism $C \lra \injc(C)$ is split. \textcolor{purple}{For any indecomposable $C$ over a left hereditary $\Lambda$, noetherian or not, the homomorphism $C \lra \injc(C)$ is either the zero map or an isomorphism.}
\end{proposition}

\begin{proof}
 $Tr(\mathscr{I}, C)$, being the sum of the images of injectives, is the image of a direct sum of such images. 
 Since $\Lambda$ is noetherian, the direct sum is injective. Using the hereditary property, we have that the sum of the images, i.e., the trace of injectives in $C$, is injective. In view of Lemma~\ref{L:trace}, the short exact sequence~\eqref{Eq:cotorsion-def} is split-exact. \textcolor{purple}{To prove the second assertion, use again the fact that the image of an injective module is injective and hence a direct summand of $C$.}
\end{proof}

\textcolor{purple}{Having described the functor $\injc^{-1}$ as the trace of injectives, we want to describe an important property of the cotorsion functor $\injc$. To this end, we introduce} 

\begin{definition}
\textcolor{purple}{A quotient functor $\mathfrak{r}$ of the identity functor is a \texttt{coradical} if it is a radical on the opposite category, i.e., if $\mathfrak{r}\mathfrak{r}^{-1} = 0$.}
\end{definition}

In line with the formal duality between torsion and cotorsion, one may ask if there is an analog of Proposition~\ref{T:residual-torsion}. The next result answers this question in the positive.

\begin{proposition}\label{P:cotorsion-free-submodule}
 $\injc$ is a coradical, i.e., $\injc(\injc^{-1}(C)) = \{0\}$ for any $C$.   \end{proposition}

\begin{proof}
\textcolor{purple}{By the definition of cotorsion,} $\injc(Tr(\mathscr{I}, C)) = \{0\}$ if and only if any module homomorphism  $g : \Lambda \to Tr(\mathscr{I}, C)$ factors through an injective. Choose one such $g$ and let $i : Tr(\mathscr{I}, C) \to C$ be the inclusion map. 
The proof of Lemma~\ref{L:trace} shows that, since $ig(1) = g(1) \in Tr(\mathscr{I}, C)$, $ig$ factors through some injective $J$
%
and we have a commutative diagram 
 \[
\xymatrix
	{
	\Lambda \ar[r]^{g} \ar[rd]
	& Tr(\mathscr{I}, C) \ar@{>->}[r]^{i}
	& C 
\\
	& J \ar[ru] \ar@{-->}[u]
	}
\] 
of solid arrows. By Lemma~\ref{L:trace}, the northeastern map factors through $i$, giving rise to a dashed arrow and a commutative triangle on the right-hand side. Since $i$ is monic, the triangle on the left commutes, too, showing that $g$ factors through an injective. 
\end{proof} 

\begin{remark}\label{R:q-not-id}
 The reader should not think that $\injc^{2} = \injc$. In plain terms, there may be nonzero maps $ \Lambda \to \injc(C)$ factoring through injectives. An example, based on a duality argument, will be provided by Remark~\ref{R:q-not-idem}. 
\end{remark}

Having a definition of the cotorsion module of a module, we can now define a notion of cotorsion module.

\begin{definition}\label{D:cotorsion-mod}
 A module $C$ is said to be a \texttt{cotorsion module} if the canonical surjection $(\Lambda, C) \lra (\overline{\Lambda, C})$ is an isomorphism. Equivalently, no nonzero map $\Lambda \to C$ factors through an injective, or, equivalently, the map $(\iota, C)$ in \eqref{cotorsion-def} is the zero map.
\end{definition}

\begin{remark}\label{R:cot-Matlis}
 As we mentioned before, while we could not find any prior definition of the cotorsion module of a module, there have been several attempts to define a more narrow notion of a cotorsion module. This has been done in various degrees of generality. In the case 
 $\Lambda$ is a commutative domain, Matlis~\cite{Mat} calls $C$ a cotorsion module if
$\Hom(Q,C) = \{0\} = \Ext^{1}(Q,C)$, where~$Q$ is the quotient field of 
$\Lambda$. In this case, $Q = I$ is the injective envelope of~$\Lambda$ and therefore the map $(\iota, C)$ from~\eqref{cotorsion-def} is zero. Thus a cotorsion module in the sense of Matlis is a cotorsion module in the sense of 
Definition~\ref{D:cotorsion-mod}. In fact, in this case, $C \simeq \injc(C) \simeq \Ext^{1}(\Sigma \Lambda, C)$, as can be seen by from the sequence~\eqref{cotorsion-def}. 
\end{remark}

\begin{proposition}
 \textcolor{purple}{The class of cotorsion modules is closed under submodules and products.}
\end{proposition}

\begin{proof}
 \textcolor{purple}{The class of cotorsion modules coincides with the torsion-free class of the preradical $\injc^{-1}$. The result now follows from  Proposition~\ref{P:tf-sub-prod}.}
\end{proof}

Having defined cotorsion modules, we proceed to define cotorsion-free modules.

\begin{definition}\label{D:cotorsion-free}
 The module $C$ is said to be \texttt{cotorsion-free} if the cotorsion module of $C$ is zero:
\[
\injc{(C)} = (\overline{\Lambda, C}) = \{0\}
\]
Equivalently, any map $\Lambda \to C$ factors through an injective, or, equivalently, the map $(\iota, C)$ in \eqref{cotorsion-def} is epic.
\end{definition}

Clearly, any injective module is cotorsion-free. It is also clear that if a module is both cotorsion and cotorsion-free, then it is the zero module.

\begin{proposition}
\textcolor{purple}{The class of cotorsion-free modules is closed under quotient modules and coproducts.}
\end{proposition}

\begin{proof}
 \textcolor{purple}{The class of cotorsion-free modules coincides with the torsion class of the preradical $\injc^{-1}$. The result now follows from  Proposition~\ref{P:t-quot-coprod}.}
\end{proof}

\begin{remark}
\textcolor{purple}{This is a companion to Remark~\ref{R:t-not-tt}. In general, the class of cotorsion-free modules is not closed under extensions. Otherwise, this class would be the torsion class of a torsion theory in the sense of Dickson. As we just remarked, this class coincides with the torsion class of $\injc^{-1}$. In particular, that would make $\injc^{-1}$ a radical, which is equivalent to saying that $\injc$ is idempotent, contradicting Remark~\ref{R:q-not-id}.} 
\end{remark}

It is a general property of coradical functors that the corresponding cotorsion-free class\footnote{Strictly speaking, this should be called the \texttt{precotorsion-free} class.} is a coreflective subcategory. In our case, this means that 
$\injc^{-1}$, called a coreflector and viewed as a functor from the category of all modules to the category of cotorsion-free modules (but not all modules!) is right adjoint to the inclusion functor. Thus, we have 

\begin{proposition}\label{P:coreflector-limits}
The functor $\injc^{-1}$ from the category of all modules to the category of cotorsion-free modules preserves all limits. In particular, the class of cotorsion-free modules has all limits.  \qed
\end{proposition}

Lemma~\ref{L:trace} and Proposition~\ref{P:cotorsion-free-submodule} justify the following 

\begin{definition}
 The trace in $C$ of the class of injective $\Lambda$-modules will be called the
 cotorsion-free submodule of $C$. 
\end{definition}

Thus a module is cotorsion-free if and only if it coincides with its cotorsion-free submodule.

Summarizing the foregoing discussion, we rewrite the defining short exact sequence for the cotorsion quotient module of $C$ using the trace of the class~$\mathscr{I}$ of injectives in~$C$:
\begin{equation}\label{Eq:q-via-injective-trace}
\xymatrix
	{
	0 \ar[r]
	& Tr(\mathscr{I}, C) \ar[r]
	& C \ar[r]
	& \injc(C) \ar[r]
	& 0
	}
\end{equation}
Recalling the definition of the left cosatellite of a contravariant 
functor~\cite[Section~5]{MR-1} and the fact that the Hom functor is left-exact (in fact, we only need the half-exactness), we observe that $Tr(\mathscr{I}, C)$ is nothing but the left cosatellite of the functor $\Hom(\blank, C)$ evaluated on 
$\Lambda$, i.e., 
\[
Tr(\mathscr{I}, C) \simeq C_{1}(\blank, C)(\Lambda)
\]
The same is expressed by the short exact sequence 
\begin{equation}
 0 \lra C_{1}(\blank, C)(\Lambda) \lra C \lra C/Tr(\mathscr{I}, C) \lra 0
 \end{equation}
which is similar to the short exact sequence~\eqref{Eq:Rej}. In multiplicative notation, the similarity with~\eqref{Eq:Rej-mult} becomes even more apparent:
\begin{equation}\label{Eq:q=Tr-inverse}
 \injc(C) = (Tr(\mathscr{I}, C))^{-1} \quad \text{and} \quad 
\injc^{-1}(C) = C_{1}(\blank, C)(\Lambda).
\end{equation}
\smallskip

\begin{proposition}
 The functor $Tr(\mathscr{I}, C) \simeq C_{1}(\blank, C)(\Lambda)$, viewed as a functor of~$C$ from all modules to cotorsion-free modules, preserves all limits.  
\end{proposition}

\begin{proof}
Follows from Proposition~\ref{P:coreflector-limits}. 
\end{proof}

\begin{remark}\label{R:EJ}
 Enochs and Jenda~\cite[Definition 5.3.22]{EJ2000} say that a module $M$ over an arbitrary ring is cotorsion if $\Ext^{1}(F,M) = \{0\}$ for any flat module $F$. They remark that that definition\footnote{\textcolor{purple}{That definition is already present in~\cite[p. 46]{J-77}.}} generalizes the definitions of Harrison~\cite{Har} and Warfield~\cite{W} and agrees with that of Fuchs~\cite{Fuchs} but differs from the definition of Matlis mentioned above. The last observation is easy to explain: injective modules are cotorsion in the sense of Enochs and Jenda, but in general such modules are not  cotorsion in the sense of Matlis. In fact, as our definition shows, a better term for the cotorsion modules in the sense of Enochs-Jenda might have been ``cotorsion-free modules'' but, even under this moniker, among such modules there would be modules which are not cotorsion-free in the sense of Definition~\ref{D:cotorsion-free}. For example, Enochs and 
 Jenda~\cite[Lemma 5.3.23]{EJ2000} remark that every pure injective is cotorsion in their sense.\footnote{\textcolor{purple}{That comment is also present in~\cite[p. 46]{J-77}.}} We shall now show that there are pure injectives which are not cotorsion-free in the sense of Definition~\ref{D:cotorsion-free}.\footnote{The authors are grateful to Gena Puninski for a helpful comment leading to this example.} Let $\Bbbk$ be a field, $\Lambda : = \Bbbk[[X]]$, and $C := \Lambda$. Then~\cite[Proposition~1]{ZHZ} $\Bbbk[[X]]$ is pure injective as a module over 
itself. As $\Bbbk[[X]]$ is a PID, any injective is divisible and therefore the  trace of the injectives in $\Bbbk[[X]]$ consists of divisible elements. Since  $\Bbbk[[X]]$ is obviously not divisible, the injective trace is properly contained in $\Bbbk[[X]]$, and therefore $C$ is not cotorsion-free. In fact, since 0 is the only divisible element in $\Bbbk[[X]]$, the latter is cotorsion. Thus the class of cotorsion 
 $\Bbbk[[X]]$-modules in the sense of Enochs-Jenda contains both cotorsion-free modules (e.g., injectives) and cotorsion modules (e.g., $\Bbbk[[X]]$).
\end{remark}

The next two results deal with the vanishing and the exactness properties of the cotorsion functor. Since injectives are cotorsion-free, Corollary~\ref{C:q-epi} yields 

\begin{proposition}
\textcolor{purple}{If $C$ is injective, then $\injc(D) = 0$ for any quotient module $D$ of~$C$. In particular, all cosyzygy modules in an injective resolution of a module are cotorsion-free. \qed}
\end{proposition}

\begin{proposition}\label{P:q=0}
 The following conditions are equivalent:
 
\begin{itemize}
 \item[a)] $\injc$ preserves monomorphisms,
\smallskip 
 \item[b)] $\injc$ is the zero functor,
 \smallskip 
 \textcolor{purple}{\item[c)] $\injc(\Lambda) \simeq 0$},
 \smallskip 
 \item[d)] $\Lambda$ is left selfinjective.
 \smallskip 
\end{itemize}
\end{proposition}

\begin{proof}
 Given any left $\Lambda$-module $C$ choose a monomorphism 
 $C \to J$ with~$J$ injective. Assuming $\injc$ preserves monomorphisms,  we have a monomorphism $\injc (C) \to  \injc (J)$. Since $J$ is injective,  $\injc (J) = \{0\}$ and therefore $\injc (C) = \{0\}$. This proves a) $\Rightarrow$ b). The converse is trivial. \textcolor{purple}{The implication b) $\Rightarrow$ c)  is also trivial. 
 Now assume that 
 $\injc(\Lambda) = 0$. This implies that the identity map on 
 $\Lambda$ factors through an injective, making  $\Lambda$ injective. This proves c) $\Rightarrow $ d).  The implication d) $\Rightarrow $ b) is immediate from the definition of $\injc$.}
\end{proof}

 \textcolor{purple}{The just proved result allows to prove the nonexistence of a functorial splitting for the cotorsion functor.} 

\begin{corollary}
 \textcolor{purple}{Suppose the cotorsion splits off as a functor, i.e., the exact sequence $0 \lra \injc^{-1} \lra \mathbf{1} \lra \injc \lra 0$ splits. Then $\injc =0$, i.e., $\Lambda$ is left selfinjective.}
 \end{corollary}
 
 \begin{proof}
 \textcolor{purple}{Under the assumption, $\injc$ becomes a subfunctor of $\mathbf{1}$ and hence preserves monomorphisms. The result follows.}
\end{proof}

Now we want to examine the properties of the cotorsion functor 
$\injc$ when the injective envelope of $\Lambda$ is projective.

\begin{proposition}
Suppose the injective envelope of $\prescript{}{\Lambda}\Lambda$ is projective. Then
 \begin{enumerate}
 
 \item $\injc \simeq \Ext^{1}(\Sigma \Lambda, \blank)$.
 \smallskip
 
 \item If $0 \to C' \to C \to C'' \to 0$ is an exact sequence, then \textcolor{purple}{there is a homomorphism $\delta : (\Sigma \Lambda, C'') \to \injc(C')$ such that} the induced sequence
\[
0 \to (\Sigma \Lambda, C') \to (\Sigma \Lambda, C) \to (\Sigma \Lambda, C'') \overset{\textcolor{purple}{\delta}}\to \injc(C') \to \injc(C) \to \injc(C'') \to 0
\]
is exact. In particular, $\injc$ is right-exact.
\smallskip

\item The natural transformation $\injc(\mathbf{1} \to \injc) = \injc \to \injc^{2}$ is an isomorphism, i.e., \textcolor{purple}{$\injc$ is an idempotent coradical and thus the cotorsion-free class of $\injc$ is the torsion class of a torsion theory. That torsion theory is hereditary, i.e., its torsion class is closed under submodules.}
\smallskip
 
\item The cotorsion-free class of $\injc$ is closed under extensions.
\smallskip

\item \textcolor{purple}{The cotorsion class of $\injc$ is closed under extensions.}
\end{enumerate}
\end{proposition}
 
 \begin{proof}
 (1) Map the short exact sequence 
 \[
0 \lra \Lambda \lra I \lra \Sigma \Lambda \lra 0
\]
into an arbitrary module and pass to the long exact sequence.

(2) Follows from \textcolor{purple}{the definition of $\injc$ and} the snake lemma. \textcolor{purple}{$\delta$ is just the connecting homomorphism.}
 
(3) By (2), $\injc$ is right-exact. Thus the short exact sequence
 \[
 0 \lra I(\Lambda, C) \lra C \lra \injc(C) \lra 0
 \]
 gives rise to the exact sequence
 \[
 \injc(I(\Lambda, C)) \lra \injc(C) \lra \injc^{2}(C) \lra 0.
 \]
By Proposition~\ref{P:cotorsion-free-submodule}, $\injc(I(\Lambda, C)) \simeq  \{0\}$, whence the first claim. \textcolor{purple}{Thus $\injc$ is an idempotent coradical, which, by duality, makes $\injc^{-1}$ an idempotent radical, and hence the torsion class of a torsion theory. That torsion class is the class of cotorsion-free modules. The torsion-free class is precisely the class of cotorsion modules. By~\cite[VI, Proposition 3.2]{S-75}, a torsion theory is hereditary if and only if its torsion-free class is closed under injective envelopes. Since injective modules are cotorsion-free, this condition is satisfied. }
 
(4) Apply $\injc$ to a short exact sequence $0 \to C' \to C \to C'' \to 0$. The resulting exact sequence $\injc(C') \to \injc(C) \to \injc(C'') \to 0$ shows that if the end terms vanish, then so does the middle term.

\textcolor{purple}{(5) This is a general property of torsion-free classes of torsion theories~(\cite[VI, Proposition 2.2]{S-75}).}
\end{proof}

\newcommand{\fpr}{\mathrm{fp}}

%
%
%
%
%



\section{Torsion, cotorsion, and the Auslander-Gruson-Jensen functor}\label{S:duality}

Throughout this section $\Mod(\Lambda)$ will denote the category of all right 
$\Lambda$-modules, while $\mod(\Lambda)$ will stand for the full subcategory determined by the finitely presented modules.

Recall that a covariant functor $F:\Mod(\Lambda^{op})\to \ab$ is finitely presented if there are modules $X,Y$ and a sequence of natural transformations \[
(Y,\blank) \lra (X,\blank) \lra F \lra 0
\]
such that for any left module $M$, the sequence of abelian groups 
\[
(Y,M) \lra (X,M) \lra F(M) \lra 0
\]
is exact.  Notice that a finitely presented functor is automatically additive. Given finitely presented functors $F$ and $G$, it is easily verified that the natural transformations between $F$ and $G$ form an abelian group. 

\begin{theorem}[\cite{A66}, Theorem 2.3]\label{T:afpthm}
Let $\fp(\Mod(\Lambda^{op}),\ab)$ denote the category of all finitely presented covariant functors together with natural transformations between them. Then:
\begin{enumerate}
\item $\fp(\Mod(\Lambda^{op}),\ab)$ is abelian.  A sequence of finitely presented functors is exact if and only if it is exact componentwise.
\smallskip
\item The projectives are precisely the representable functors $(M,\blank)$.
\smallskip
\item Every finitely presented functor $F\in \fp(\Mod(\Lambda^{op}),\ab)$ has a projective resolution of the form 
\[
0 \lra (Z,\blank) \lra (Y,\blank) \lra (X,\blank) \lra F \lra 0
\]
\end{enumerate}
\end{theorem} \qed

Moreover, since $\Mod(\Lambda^{op})$ has enough projectives, 
$\fp(\Mod(\Lambda^{op}),\ab)$ has enough injectives. This result is due to Ron Gentle and appears in~\cite{Gen}, where the existence of injectives is shown in Proposition 1.4 and the discussion following that proposition.  The fact that 
$\fp(\Mod(\Lambda^{op}),\ab)$ has enough injectives allows one to compute right derived functors.

In addition to $\fp(\Mod(\Lambda^{op}),\ab)$, we will be looking at the category $(\mod(\Lambda),\ab)$ of additive functors $\mod(\Lambda) \to \ab$. Both categories contain certain functors which are fundamental to the study of model theory of modules. In order to understand these functors, we first remark that the notion of finitely presented functor also makes sense when the domain is just an additive category. 
%
Thus the category $\fp(\mod(\Lambda),\ab)$ consists of all functors $F:\mod(\Lambda)\to \ab$ for which there are finitely presented modules $X,Y$ and a presentation 
\begin{equation}\label{E:presentation}
 (Y,\blank) \lra (X,\blank) \lra F\lra 0
\end{equation}
If the module $X$ is finitely presented, then the representable functor 
\[
(X,\blank):\Mod(\Lambda) \lra \ab
\]
commutes with filtered colimits and is thus the colimit extension of its restriction to finitely presented modules.   As a result, given any finitely presented functor  $F\in \fp(\mod(\Lambda),\ab)$ with presentation~\eqref{E:presentation},
the functor $\overset{\to}{F}:\Mod(\Lambda)\to \ab$ has a presentation 
\[
(Y,\blank) \lra (X,\blank) \lra \overset{\to}{F}\to 0
\] 
Hence $\overset{\to}{F}$ is finitely presented as a functor on the large module category $\Mod(\Lambda)$.  Thus $\fp(\mod(\Lambda),\ab)$ is a subcategory of both the functor category $(\mod(\Lambda),\ab)$ and the functor category 
$\fp(\Mod(\Lambda),\ab)$, the latter via $F\mapsto\overset{\to}{F}$.    

As these functors sit inside the three different functor categories $\fp(\mod(\Lambda),\ab)$, $(\mod(\Lambda),\ab)$, and $\fp(\Mod(\Lambda),\ab)$, where two of these categories consist of finitely presented functors with different domain categories and the third has both finitely presented and non-finitely presented functors, the term finitely presented may become confusing.  A first step to escape this quandary is to use the equivalence between 
$(\mod(\Lambda),\ab)$ and the category of functors on $\Mod(\Lambda)$ that commute with filtered colimits.  This allows us to view both $\fp(\Mod(\Lambda),\ab)$ and $(\mod(\Lambda),\ab)$ as consisting of functors on $\Mod(\Lambda)$.    

Once this convention is taken, we can identify the functors in 
$ \fp(\mod(\Lambda),\ab)$ as being the intersection or the categories $\fp(\Mod(\Lambda),\ab)$ and $(\mod(\Lambda),\ab)$.   We will use the following terminology. 
A  functor $F:\Mod(\Lambda)\to \ab$ is called a \texttt{pp-functor} if there exist finitely presented modules $X$ and $Y$ and a presentation 
\[
(Y,\blank) \lra (X,\blank) \lra F\to 0
\]  
With this terminology set, the full subcategory of $\fp(\Mod(\Lambda^{op}),\ab)$ consisting of the pp-functors is equivalent to the functor category $\fp(\mod(\Lambda^{op}),\ab)$, and the full subcategory of $(\mod(\Lambda),\ab)$ consisting of all pp-functors is equivalent to the functor category $\fp(\mod(\Lambda),\ab)$.  In addition, these two full subcategories are abelian and their inclusions are exact.

The Auslander-Gruson-Jensen (AGJ) duality, discovered by Gruson and Jensen in~\cite{GJ} and independently by Auslander in~\cite{A84}, is a pair of exact contravariant functors
 \[
\xymatrix
	{
	\fp(\mod(\Lambda^{op}),\ab) \ar@/^2pc/[rrr]^{D} 
	&
	&
	& \fp(\mod(\Lambda),\ab) 	\ar@/^2pc/[lll]^{D}
	}
\]
\textcolor{purple}{defined by $DF(B) := (F, B \otimes \blank)$ on functors and by $D \alpha (B) := (\alpha, B \otimes \blank)$ on natural transformations}. It has the following properties:
\begin{enumerate}
\item If $X$ is a finitely presented left module then 
\[
D(X,\blank) \simeq \blank\otimes X \qquad\text{and} \qquad D(\blank\otimes X) \simeq (X,\blank)
\]
\item If $X$ is a finitely presented right module then 
\[
D(X,\blank) \simeq X\otimes\blank \qquad\text{and} \qquad D(X\otimes\blank) \simeq (X,\blank)
\] 
\end{enumerate}

There is a functor 
\[
D_A:\fp(\Mod(\Lambda^{op}),\ab) \lra (\mod(\Lambda),\ab)
\]
 defined by 
 \[
 D_A := R_0 (\epsilon \circ w), 
 \]
 where $\epsilon$ is the tensor embedding
\[
 \epsilon : \Mod(\Lambda^{op}) \lra (\mod(\Lambda),\ab) : M \mapsto \blank 
 \otimes M
\]
and $w$ is the defect functor~\cite[Proposition 6.1]{MR-1}.
The functor $D_{A}$ is contravariant, exact, and for any representable functor $(M,\blank)$ 
 \[
 D_A(M,\blank)=\blank \otimes M
 \]
 As shown in~\cite[Proposition~9]{RD}, the functor $D_A$ is completely determined by these properties.

\begin{theorem}[~\cite{RD}, Theorems 23 and 29]\label{T:DA}
The functor 
\[D_A:\fp(\Mod(\Lambda^{op}),\ab) \lra (\mod(\Lambda),\ab)
\] 
admits a left adjoint $D_L$ and a right adjoint $D_R$, both of which are fully faithful. The functors $D_R$ and $D_A$ restrict to the Auslander-Gruson-Jensen duality $D$ on the full subcategories of pp-functors. \qed
 \end{theorem} 
 
 The foregoing statement is part of the following diagram of functors
 
 \begin{center}\begin{tikzpicture}
\matrix (m) [ampersand replacement= \&,matrix of math nodes, row sep=4em, 
column sep=4em,text height=1.5ex,text depth=0.25ex] 
{\fp(\Mod(\Lambda^{op}),\ab)\&\&(\mod(\Lambda),\ab)\\
\&\&\\
\&\Mod(\Lambda^{op})\&\\}; 
\path[->,    font=\scriptsize]
(m-3-2) edge[bend right=20] node[right=5]{$\epsilon$}(m-1-3)
(m-1-3)edge node[right]{$\ev_\Lambda$}(m-3-2)
(m-3-2)edge[bend left=20] node[right=1]{$R_0\epsilon$}(m-1-3);
\path[->,  font=\scriptsize]
(m-1-1) edge node[auto]{$w$}(m-3-2)
(m-3-2) edge[bend left=15] node[right=1]{$\textsf{Y}$}(m-1-1)
(m-3-2) edge[bend right=15] node[right=5]{$L^0\textsf{Y}$}(m-1-1);
\path[->,     font=\scriptsize]
(m-1-1) edge node[auto]{$D_A$}(m-1-3);
\path[->,    font=\scriptsize]
(m-1-3) edge[bend right=15] node[above]{$D_L$}(m-1-1);
\path[->,      font=\scriptsize]
(m-1-3) edge[bend left=15] node[above]{$D_R$}(m-1-1);
\end{tikzpicture}\end{center}

Before proceeding, we once again return to an issue with notation.\label{Page:not-mod-inj} Recall that for any functor $F$, the projective stabilization is denoted by $\underline{F}$ and the injective stabilization is denoted by $\overline{F}$. One of the interesting observations from~\cite{AB} is that the projective stabilization of the functor $(B,\blank)$ is the functor $(\underline{B,\blank})$, which sends any module $C$ to the abelian group $(\underline{B, C})$, known as Hom modulo projectives.  In other words, if $F = (B,\blank)$, then we have an isomorphism
\[
\underline{F} \simeq (\underline{B,\blank})
\]       
Continuing with the functor $F = (B,\blank)$, its injective stabilization $\overline{F}$ is $0$ because~$F$ is left-exact.  Therefore, the injective stabilization of $(B,\blank)$ does not return the functor $(\overline{B, \blank})$, which sends any module $C$ to the abelian group $(\overline{B,C})$, known as Hom modulo injectives. Even though $(\overline{B,\blank})$ is not the injective stabilization of $(B,\blank)$, it is injectively stable as it clearly vanishes on injectives. As we remarked before, this functor is also finitely presented: just apply the contravariant Hom functor to a cosyzygy sequence of $B$.

\textcolor{purple}{At this point it is convenient to recall the active/inert terminology introduced earlier in this paper. The passage from torsion to cotorsion was motivated by a somewhat imprecise formal duality between the injective stabilization $A \ot \blank$ of the tensor product and the projective stabilization $(\overline{\blank, C})$ of the contravariant Hom functor. A more precise statement would refer to the \texttt{active} injective stabilization of the tensor product and the \texttt{active} projective stabilization of the contravariant Hom. Following a suggestion by George Janelidze, we denote the latter by 
$(\overset{\leftharpoondown}{\blank, C})$.}\footnote{\textcolor{purple}{A similar harpoon notation, in place of the underline, can be used to denote the \texttt{projective stabilization} of the covariant Hom. With all these conventions in place, we can finally avoid the confusion with the stabilizations and Hom modulo projectives or injectives.}}
 
\textcolor{purple}{On the other hand, the injective torsion was defined as  $\blank \ot \Lambda$, i.e., as the \texttt{inert} injective stabilization of the tensor product with $\Lambda$, and the cotorsion was defined as $(\overset{\leftharpoondown}{\Lambda, \blank})$, i.e., as the \texttt{inert} projective stabilization of Hom from $\Lambda$.}

One of the major reasons for looking at the \textcolor{purple}{\texttt{inert} injective stabilization of the tensor product is that this notion is dual, via the extension of the Auslander - Gruson - Jensen duality by Dean - Russell, to that of the \texttt{inert} projective stabilization of the contravariant Hom functor. Before stating that result we remind the reader that the latter is finitely presented and so the functor $D_{A}$ can be applied to it, and that $D_{A}$ sends representable functors to univariate tensor product functors.}


\begin{theorem}\label{T:cotorthm}
For any module $B$ 
\[
D_A\textcolor{purple}{(\overset{\leftharpoondown}{B, \blank})} \simeq \blank \ot B.\footnote{\textcolor{purple}{By the definition of $D_{A}$, the left-hand side of this formula is a functor defined on finitely presented modules. Thus the right-hand side should also be viewed as a functor with the same domain. However, as we have seen, $\blank \ot B$ preserves filtered colimits and is thus the colimit extension of its restriction to finitely presented modules.}}
\]
\end{theorem}

\begin{proof}
As seen above, the cosyzygy sequence $0 \to B \to I \to \Sigma B \to 0$  yields a presentation 
\[
(I, \blank) \lra (B, \blank) \lra \textcolor{purple}{(\overset{\leftharpoondown}{B, \blank})} \lra 0.
\]
Applying the contravariant exact functor $D_A$ yields an exact sequence of functors in $(\mod(\Lambda),\ab)$
\[
0 \lra D_A\textcolor{purple}{(\overset{\leftharpoondown}{B, \blank})} \lra \blank \otimes B \lra \blank \otimes I,
\]
which establishes that $D_A\textcolor{purple}{(\overset{\leftharpoondown}{B, \blank})} \simeq  \blank \ot B$.      
\end{proof}

\begin{corollary}\label{C:AGJ-q-c}
The Auslander-Gruson-Jensen functor sends the cotorsion functor on left (right) modules to the injective torsion functor on right (left) modules. In short, 
\[
D_A(\injc) \simeq \injt.
\]
Equivalently,
\[
D_A(Tr(\mathscr{I}, \blank)^{-1}) \simeq \overset{\lra}{rej}(\blank, \mathscr{F}).
\]
\end{corollary} 

\begin{proof}
 For the first isomorphism, set $B := \Lambda$. The second claim follows from 
 Proposition~\ref{P:s=colim-ext} and~\eqref{Eq:q=Tr-inverse}
\end{proof}

\begin{corollary}
\textcolor{purple}{$\Lambda$ is one-sided absolutely pure if and only if every pure injective $\Lambda$-module on the other side is cotorsion-free.}
\end{corollary}

\begin{proof}
Proposition~\ref{P:s-epi} shows that $\Lambda$ is left absolutely pure if and only if $\injt$ is the zero functor on right modules. By the preceding corollary, this is equivalent to $D_A(\injc)=0$. By~\cite[Proposition 19]{RD}, $D_A(\injc)=0$ if and only if $\injc$ vanishes on all pure injective left $\Lambda$-modules.
\end{proof}

\begin{remark}
 As we mentioned in Remark~\ref{R:EJ}, Enochs and Jenda show that any pure injective is cotorsion in their sense. From the point of view of our definitions, one may consider replacing their term ``cotorsion'' by our term ``cotorsion-free''. The just proved corollary then provides a result that may be considered as a ``replacement'' of the result of Enochs and Jenda. 
\end{remark}

%

If $M$ is a pure injective left $\Lambda$-module, then 
$\blank\otimes M$ is injective in the functor category 
$(\mod(\Lambda),\ab)$. The functor $D_L$, which is the left adjoint to $D_A$, sends any injective $\blank\otimes M$ to the representable 
$(M,\blank)$. Because $D_L$ is also right-exact, one can easily show the following.  

\begin{proposition}For any pure injective left $\Lambda$-module $M$, 
\[
\textcolor{purple}{(\overset{\leftharpoondown}{M, \blank})} \simeq D_L( \blank\ot M)
\]
\end{proposition}

\begin{proof}
For the pure injective module $M$ take any monomorphism 
$0\to M\to I$ with $I$ injective. By applying the right-exact contravariant functor $D_L$ to the exact sequence 
\[
0 \lra \blank \ot M \lra \blank\otimes M \lra  \blank\otimes I
\] 
we have the exact sequence of functors 
\[
(I,\blank) \lra (M,\blank) \lra D_L ( \blank\ot M) \lra 0
\]
It follows that $D_L(\blank\ot M) \simeq \textcolor{purple}{(\overset{\leftharpoondown}{M, \blank})}$.  
\end{proof}

\begin{corollary}
If $_{\Lambda} \Lambda$ is pure injective, then $\injc \simeq D_L(\injt)$.  
\qed\end{corollary}

In the case when $\Lambda$ is an algebra over a commutative ring $R$, the connections between torsion and cotorsion can be made more pointed by utilizing the generalized Auslander-Reiten formula~\cite[Proposition~9.7]{MR-1}. Let $\mathbf{J}$ be an injective $R$-module and $D_{\mathbf{J}} := \Hom_{R} (\blank, \mathbf{J})$. Then
\[
 D_{\mathbf{J}}(A\ot B) \simeq  \textcolor{purple}{(\overset{\!\!\!\!\!\!\!\!\!\!\!\!\!\!\!\leftharpoondown}{B, D_{\mathbf{J}}(A)})}, 
 \]
where $A$ is an arbitrary right $\Lambda$-module and $B$ is an arbitrary left $\Lambda$-module. Specializing to the case $B = \prescript{}{\Lambda}{\Lambda}$, we immediately have

\begin{proposition}\label{P:Ds=qD}
 In the above notation, 
 \[
 D_{\mathbf{J}} \circ \injt \simeq \injc \circ D_{\mathbf{J}},
 \]
 i.e., for each injective $R$-module $\mathbf{J}$ and each right $\Lambda$-module $A$, we have an isomorphism $D_{\mathbf{J}}(\injt(A)) \simeq \injc(D_{\mathbf{J}}(A))$ which is functorial in $A$. \qed
\end{proposition}

\begin{corollary}\label{C:exchange-formula}
 Let $\Lambda$ be an arbitrary ring, $R := \Z$, and, for any right $\Lambda$-module~$A$, let $A^{+} := \Hom_{\Z}(A, \Q/\Z)$ be the character module of $A$.
 Then 
\[
\injt(A)^{+} \simeq \injc(A^{+}) 
\] 
\end{corollary}
\qed

Since character modules are pure injective, we have
\begin{corollary}
 \textcolor{purple}{For any module $A$, the module $\injc(A^{+})$ is pure injective. \qed}
\end{corollary}

\textcolor{purple}{The next application was suggested by the referee.}

\begin{corollary}
 \textcolor{purple}{If $A$ is pure injective, then so is $\injc(A)$.}
\end{corollary}

\begin{proof}
\textcolor{purple}{For any module $A$, the canonical map $A \to A^{++}$ is pure. Assuming that~$A$ is pure injective, we have that this embedding splits and $\injc(A)$ becomes a direct summand of the pure injective $\injc(A^{++})$. The result now follows.}
\end{proof}

\begin{remark}\label{R:q-not-idem}
 As yet another application of Proposition~\ref{P:Ds=qD}, we can now show that, in general, $\injc^{-1}$ is not a radical or, equivalently, $\injc$ is not idempotent. This can be seen from the following example. Let $\Lambda$ be a commutative local finite-dimensional $\Bbbk$-algebra over a field $\Bbbk$. Then $D_{\Bbbk}$ is a duality on the category of finite-dimensional $\Lambda$-modules. Applying~$D_{\Bbbk}$ to the example from Remark~\ref{R:low-torsion}, we have the desired example. Details are left to the reader.
\end{remark}

\textcolor{purple}{In preparation for the next section, we are going to establish yet another connection between torsion and cotorsion by showing that $\injt$ is a subfunctor of $(\injc(\prescript{}{\Lambda}{\Lambda}), \blank)$, where both functors are defined on right $\Lambda$-modules. That this makes sense for the latter is guaranteed by Lemma~\ref{L:q-L-bimodule}, showing that $\injc(\prescript{}{\Lambda}{\Lambda})$ is a bimodule.}

\textcolor{purple}{We begin by recalling the natural transformation 
$\mu_{B} : \blank \otimes B \lra (B^{\ast}, \blank)$ from Lemma~\ref{L:mu-on-proj} (with an obvious change of side). The map  
\[
\mu_{B}(M) : M \otimes B \lra (B^{\ast}, M)
\]
is defined by $\mu_{B}(M)(m \otimes b)(l) :=  m(bl)$ for any $b \in B$, $m \in M$, and $l \in B^{\ast}$, where we are, once again, following the rule ``maps on the opposite side from scalars'', which we shall also do for the rest of the paper. When $B := \prescript{}{\Lambda}{\Lambda}$ we have the canonical isomorphism 
$\mu_{\prescript{}{\Lambda}{\Lambda}} : \blank \otimes \prescript{}{\Lambda}{\Lambda} \lra (\Lambda_{\Lambda}, \blank)$. To avoid overloading the notation, we denote it by simply $\mu_{\Lambda}$. By definition, $\injt$ is a subfunctor of $\blank \otimes \prescript{}{\Lambda}{\Lambda}$, and, by Lemma~\ref{L:q-L-bimodule}, 
$(\injc(\prescript{}{\Lambda}{\Lambda}), \blank)$ is a subfunctor of 
$(\Lambda_{\Lambda}, \blank)$. }


\textcolor{purple}{Choose once again a defining cosyzygy sequence 
\begin{equation}
 0 \lra \prescript{}{\Lambda}{\Lambda} \overset{\iota}\lra I \lra \Sigma\Lambda \lra 0.
\end{equation}
Passing to the corresponding univariate tensor product functors we have the defining exact sequence for $\injt$
\[
0 \lra \injt \lra \blank \otimes \prescript{}{\Lambda}{\Lambda} \overset{\blank \otimes \iota}\lra \blank \otimes_{\Lambda} I \lra \blank \otimes_{\Lambda} \Sigma\Lambda \lra 0.
\]
Applying the natural transformations 
$\mu_{\Lambda}$, $\mu_{I}$, and $\mu_{\Sigma\Lambda}$, we have a commutative diagram 
\begin{equation}\label{Eq:s-as-subfunctor}
\begin{gathered}
\xymatrix
	{
	0 \ar[r] 
	& \injt \ar[r] \ar@{.>}[d]^{\kappa}
	& \blank \otimes \prescript{}{\Lambda}{\Lambda} \ar[r]^{\blank \otimes 		\iota} \ar[d]_{\cong}^{\mu_{\Lambda}}
	& \blank \otimes _{\Lambda} I \ar[r] \ar[d]^{\mu_{I}}
	&  \blank \otimes _{\Lambda} \Sigma\Lambda \ar[r] 				\ar[d]^{\mu_{\Sigma\Lambda}}
	& 0
\\
	0 \ar[r]
	& (\injc(\prescript{}{\Lambda}{\Lambda}), \blank) \ar[r]
	& (\Lambda_{\Lambda}, \blank) \ar[r]^{(\iota^{\ast}, \blank)}
	& (I^{\ast}, \blank) \ar[r]
	& (\Sigma\Lambda^{\ast}, \blank) \ar[r]
	& 0
	}
\end{gathered}
\end{equation}
of solid arrows. By the left-exactness of the Hom functor and the universal property of kernels, $\mu_{\Lambda}$ induces a natural transformation $\kappa$, which is clearly monic. We have thus proved}

\begin{proposition}\label{P:s-to-q}
\textcolor{purple}{$\injt$ is a subfunctor of  $(\injc(\prescript{}{\Lambda}{\Lambda}), \blank)$ via $\kappa$.} \qed
\end{proposition}

\textcolor{purple}{On the other hand, starting with the short exact sequence 
\begin{equation}\label{Eq:cot-def}
 0 \lra I(\prescript{}{\Lambda}{\Lambda}, \prescript{}{\Lambda}{\Lambda}) \overset{j}\lra \Lambda_{\Lambda} \lra \injc(\prescript{}{\Lambda}{\Lambda}) \to 0, 
\end{equation}
where $I(\prescript{}{\Lambda}{\Lambda}, \prescript{}{\Lambda}{\Lambda})$ is the bimodule of endomorphisms of $\prescript{}{\Lambda}{\Lambda}$ factoring through injectives or, equivalently, this is the trace $Tr(\mathscr{I}, \prescript{}{\Lambda}{\Lambda})$ of injectives in $\prescript{}{\Lambda}{\Lambda}$. Passing to the corresponding covariant Hom functors and comparing the result with the bottom row of~\eqref{Eq:s-as-subfunctor} we have a commutative diagram 
\begin{equation}\label{Eq:beta}
\begin{gathered}
\xymatrix
	{
	0 \ar[r] 
	& \injt \ar[r] \ar@{.>}[d]^{\kappa}
	& \blank \otimes \prescript{}{\Lambda}{\Lambda} \ar[r] 
	\ar[d]_{\cong}^{\mu_{\Lambda}}
	& \injt^{-1 }\ar[r] \ar@{.>}[d]^{\beta}
	& 0
\\
	0 \ar[r]
	& (\injc(\prescript{}{\Lambda}{\Lambda}), \blank) \ar[r]
	& (\Lambda_{\Lambda}, \blank) \ar[r]^>>>>>{(j, \blank)}
	& (Tr(\mathscr{I}, \prescript{}{\Lambda}{\Lambda}), \blank) 
	}
\end{gathered}
\end{equation}
of solid arrows with exact rows. By the universal property of cokernels, $\mu_{\Lambda}$ induces  
$\beta : \injt^{-1} \lra (Tr(\mathscr{I}, \prescript{}{\Lambda}{\Lambda}), \blank)$.}


\begin{proposition}\label{P:beta-on-s-1}
\textcolor{purple}{The natural transformation $\beta : \injt^{-1} \lra (Tr(\mathscr{I}, \prescript{}{\Lambda}{\Lambda}), \blank)$ evaluates to an epimorphism on any injective right $\Lambda$-module.}
\end{proposition}

\begin{proof}
\textcolor{purple}{Since mapping into an injective is an exact functor, the bottom row of the diagram~\eqref{Eq:beta} evaluates to a short exact sequence on any injective. The result now follows from the snake lemma.}
\end{proof}
\medskip

\textcolor{purple}{Proposition~\ref{P:s-to-q} has an analog for $\injc$. To explain it, we recall the natural transformation $\tau_{F} :F(\Lambda) \otimes \blank \lra F$ (see~\cite[p.~6]{MR-1} for details), where $F$ is an additive functor on left $\Lambda$-modules. Taking $F := \injc$, we have a natural transformation $\tau_{\injc} : \injc(\prescript{}{\Lambda}{\Lambda}) \otimes \blank \lra \injc$.}

\begin{proposition}\label{P:tau-q-epic}
 \textcolor{purple}{The natural transformation $\tau_{\injc} : \injc(\prescript{}{\Lambda}{\Lambda}) \otimes \blank \lra \injc$ is epic. Thus~$\injc$ is a quotient functor of 
 $\injc(\prescript{}{\Lambda}{\Lambda}) \otimes \blank$.}
\end{proposition}

\begin{proof}
 \textcolor{purple}{Recall the projective resolution of $\injc$ from Remark~\ref{R:pdq}:
 \[
 0 \lra (\Sigma\Lambda, \blank) \lra (I,\blank) \lra 
(\prescript{}{\Lambda}{\Lambda}, \blank) \lra \injc \lra 0.
 \]
 By the naturality of $\tau$, we have a commutative square
\[
\xymatrix
	{
	(\prescript{}{\Lambda}{\Lambda}, \prescript{}{\Lambda}{\Lambda}) \otimes \blank 
	\ar@{->>}[r] \ar[d]_{\tau_{(\Lambda, \blank)}}^{\cong}
	& \injc(\prescript{}{\Lambda}{\Lambda}) \otimes \blank \ar[d]^{\tau_{\injc}}
\\
	(\prescript{}{\Lambda}{\Lambda}, \blank) \ar@{->>}[r]
	& \injc
	}
\]
with an epic bottom arrow. Since the forgetful functor $(\prescript{}{\Lambda}{\Lambda}, \blank)$ obviously commutes with coproducts and is right-exact, $\tau_{(\Lambda, \blank)}$ is an isomorphism by~\cite[Theorem~3.5]{MR-1}). (This is also easy to see directly.) It follows that $\tau_{\injc}$ is epic.}
\end{proof}

\textcolor{purple}{Before stating an analog of Proposition~\ref{P:beta-on-s-1} for $\injc$, we draw, for future reference, a commutative diagram similar to~\eqref{Eq:s-as-subfunctor}:
\begin{equation}\label{Eq:q-as-quot-functor}
 \begin{gathered}
\xymatrix
	{
	0 \ar[r]
	& \Sigma\Lambda^{\ast} \otimes \blank \ar[r] 					\ar[d]_{\tau_{(\Sigma\Lambda^{\ast}, \blank)}}
	& I^{\ast} \otimes \blank \ar[r] \ar[d]_{\tau_{(I, \blank)}}
	& \Lambda_{\Lambda} \otimes \blank \ar[r] 
	\ar[d]^{\cong}_{\tau_{(\Lambda, \blank)}}
	& \injc(\prescript{}{\Lambda}{\Lambda}) \otimes \blank \ar[d]^{\tau_{\injc}} 		\ar[r]
	& 0
\\
	0 \ar[r]
	&(\Sigma\Lambda, \blank) \ar[r]
	& (I, \blank) \ar[r]^{(\iota, \blank)} 
	& (\prescript{}{\Lambda}{\Lambda}, \blank) \ar[r]
	& \injc \ar[r]
	& 0,
	}
\end{gathered}
\end{equation}
where the bottom row is the defining exact sequence for $\injc$. By the naturality of $\tau$, we have a commutative diagram 
\begin{equation}\label{Eq:tau-q-1}
 \begin{gathered}
\xymatrix
	{
	& Tr(\mathscr{I}, \prescript{}{\Lambda}{\Lambda}) \otimes \blank \ar[r] 		\ar[d]^{\tau_{\injc^{-1}}}
	& \Lambda_{\Lambda} \otimes \blank \ar[r] 
	\ar[d]^{\tau_{(\Lambda, \blank)}}_{\cong}
	& \injc(\prescript{}{\Lambda}{\Lambda}) \otimes \blank \ar[d]^{\tau_{\injc}} 		\ar[r]
	& 0
\\
	0 \ar[r]
	& \injc^{-1} \ar[r]^{(\iota, \blank)} 
	& (\prescript{}{\Lambda}{\Lambda}, \blank) \ar[r]
	& \injc \ar[r]
	& 0.
	}
\end{gathered}
\end{equation}
}

\begin{proposition}\label{P:tau-1-mono}
\textcolor{purple}{The natural transformation $\tau_{\injc^{-1}} : Tr(\mathscr{I}, \prescript{}{\Lambda}{\Lambda}) \otimes \blank \lra \injc^{-1}$ evaluates to a monomorphism on any projective left $\Lambda$-module.}
\end{proposition}

\begin{proof}
 \textcolor{purple}{Since tensoring with a projective is an exact functor, the top row of the diagram~\eqref{Eq:tau-q-1} evaluates to a short exact sequence on any projective. The result now follows from the snake lemma.}
\end{proof}
\medskip

\textcolor{purple}{The striking similarities between the diagrams~\eqref{Eq:s-as-subfunctor} and~\eqref{Eq:beta} and, respectively, \eqref{Eq:q-as-quot-functor} and~\eqref{Eq:tau-q-1} raise a question of whether there is a connection between the natural transformation $\mu$ and $\tau$. We are going to answer this question in the positive. To this end, we recall the operation $D$ used in the definition  of the AGJ duality. It acts on functors and natural transformations between them by the rules $DF(B) := (F, B \otimes \blank)$, where $F$ is a functor, and $D \alpha (B) := (\alpha, B \otimes \blank)$, where $\alpha$ is a natural transformation, for any right $\Lambda$-module $B$.}

\textcolor{purple}{Let $A$ be a left $\Lambda$-module. For brevity, set 
\[
\tau_{A} := \tau_{(A, \blank)} : A^{\ast} \otimes \blank \lra (A, \blank).
\]
}
\begin{proposition}\label{P:D-tau=mu}
 \textcolor{purple}{$D\tau_{A} \simeq \mu_{A}$.\footnote{\textcolor{purple}{If we set $\mu_{\blank \otimes A} := \mu_{A}$, then, in full detail, we claim that $D\tau_{(A, \blank)} \simeq \mu_{\blank \otimes A}$.}}}
\end{proposition}

\begin{proof}
 \textcolor{purple}{Given a right $\Lambda$-module $B$, the abelian group homomorphism 
 \[
 D\tau_{A}(B) : \big((A,\blank), B \otimes \blank\big) \lra 
 \big(A^{\ast} \otimes \blank, B \otimes \blank \big)
 \]
 is effected by the precomposition with $\tau_{A}$, i.e., for any 
 $\beta \in \big((A,\blank), B \otimes \blank\big)$, 
 $D\tau_{A}(B)(\beta)$ is the composition 
 \[
 A^{\ast} \otimes \blank \overset{\tau_{A}}\lra (A, \blank) \overset{\beta} \lra B \otimes \blank.
 \]
Recall the easily proved (via projective presentations and evaluation on $\Lambda$) fact that the natural transformations between univariate tensor product functors are in a bijective correspondence with the morphisms of their fixed arguments. This gives an isomorphism of abelian groups
\[
(A^{\ast} \otimes \blank, B \otimes \blank) \lra (A^{\ast}, B). 
\]
On the other hand, the Yoneda lemma yields an isomorphism 
\[
\big((A,\blank), B \otimes \blank\big) \lra B \otimes A.
\]
Putting all things together, we have a diagram
\[
\xymatrix
	{
	\big((A,\blank), B \otimes \blank \big) \ar[r]^{D\tau_{A}} 			\ar[d]_{\cong}
	& \big(A^{\ast} \otimes \blank, B \otimes \blank \big) 				\ar[d]^{\cong}
\\
	B \otimes A \ar[r]^{\mu_{A}}
	& (A^{\ast}, B).
	}
\]  
The notation-challenging, but otherwise straightforward, verification that this diagram commutes is left to the reader.}
\end{proof}

\textcolor{purple}{It is now easy to see that the diagrams~\eqref{Eq:s-as-subfunctor} and~\eqref{Eq:beta} can be obtained, respectively, from the diagrams \eqref{Eq:q-as-quot-functor} and~\eqref{Eq:tau-q-1} by applying the functor $D$. Therefore, one has an alternative way to prove Propositions~\ref{P:s-to-q} and \ref{P:beta-on-s-1} by first establishing Propositions~\ref{P:tau-q-epic} and \ref{P:tau-1-mono} and then applying Proposition~\ref{P:D-tau=mu}.}

\section{\textcolor{purple}{Torsion and cotorsion over FPE and FGE rings}}

\textcolor{purple}{The goal of this section is to show that quite a lot more can be said about torsion and cotorsion when the injective envelope of the ring is finitely presented. It is convenient to first give a name to the class of rings with this property.}  

\begin{definition}
 \textcolor{purple}{A ring $\Lambda$ is said to be left FPE if the injective envelope of $\Lambda$ viewed as a left module over itself is finitely presented.}
\end{definition}

\textcolor{purple}{Clearly, all artin algebras are FPE, so this class is quite large. In fact, any artin algebra is two-sided FPE. For the next several results we are going to assume, unless stated otherwise,  that $\Lambda$ is left FPE. Thus, all terms of the defining cosyzygy sequence }
\begin{equation}\label{Eq:def-L-I}
\textcolor{purple}{ 0 \lra \prescript{}{\Lambda}{\Lambda} \overset{\iota}\lra I \lra \Sigma\Lambda \lra 0}
\end{equation}
\textcolor{purple}{are finitely presented. Our first goal is to compute all derived functors of torsion and cotorsion over left FPE rings.}

\textcolor{purple}{In the proof of Proposition~\ref{P:fp implies s=t} we showed that, if the left $\Lambda$-module $A$ is finitely presented, then the canonical natural transformation $\mu_{A} : \blank \otimes A \lra (A^{\ast}, \blank)$ evaluates to an isomorphism on any injective right 
$\Lambda$-module. In particular, this applies to $A := I$, which makes 
$\mu_{I}$ from the diagram~\eqref{Eq:s-as-subfunctor} an isomorphism on injectives. Since $(\injc(\prescript{}{\Lambda}{\Lambda}),\blank)$ 
is left-exact, we have proved, by virtue of the analog of~\cite[Lemma~3.2, (7)]{MR-1} for right derived functors,}

\begin{theorem}
\textcolor{purple}{If $\Lambda$ is left FPE, then  
 $\kappa : \injt \lra (\injc(\prescript{}{\Lambda}{\Lambda}), \blank)$ is the zeroth right derived functor of $\injt$ on right $\Lambda$-modules.  Consequently,  
\[
R^{i}\injt \simeq \Ext^{i}(\injc(\prescript{}{\Lambda}{\Lambda}), \blank)
\]
for all integers $i$. \qed}
\end{theorem}

\begin{remark}
\textcolor{purple}{The just proved theorem allows to compute the torsion of any injective right $\Lambda$-module $J$ over a left FPE ring $\Lambda$ by utilizing the isomorphism $\injt(J) \simeq (\injc(\prescript{}{\Lambda}{\Lambda}), J)$.}
\end{remark}

\textcolor{purple}{The exact sequence $0 \lra \injt \lra \blank \otimes \prescript{}{\Lambda}{\Lambda} \overset{\blank \otimes \iota}\lra \blank \otimes_{\Lambda} I 
$ shows that $\injt$ is a finitely presented functor and the above theorem immediately tells us what its defect is.}

\begin{corollary}
 \textcolor{purple}{$w(\injt) \simeq \injc(\prescript{}{\Lambda}{\Lambda})$.} 
\end{corollary}

\begin{proof}
 \textcolor{purple}{Follows from~\cite[Proposition 3.12]{MR-1}.\footnote{This is also easy to compute directly from the copresentation of $\injt$.}}
\end{proof}

\textcolor{purple}{Since the bottom row of the diagram~\eqref{Eq:beta} evaluates to a short exact sequence on any injective, we have}
\begin{theorem}
 \textcolor{purple}{If $\Lambda$ is left FPE, then  
 $\beta : \injt^{-1} \lra (Tr(\mathscr{I}, \prescript{}{\Lambda}{\Lambda}), \blank)$ is the zeroth right derived functor of $\injt^{-1}$ on right $\Lambda$-modules.  Consequently,  
\[
R^{i}\injt^{-1} \simeq \Ext^{i}(Tr(\mathscr{I}, \prescript{}{\Lambda}{\Lambda}), \blank)
\]
for all integers $i$. \qed}
\end{theorem}

\begin{corollary}
 \textcolor{purple}{$w(\injt^{-1}) \simeq Tr(\mathscr{I}, \prescript{}{\Lambda}{\Lambda})$. \qed}
\end{corollary}

\textcolor{purple}{To determine the left derived functors of $\injt$ and $\injt^{-1}$ we don't need any assumptions on $\Lambda$. Since projectives have zero torsion, we have}

\begin{proposition}
 \textcolor{purple}{Without any assumptions on $\Lambda$, $L_{i}\injt = 0$ for all integers $i$. \qed}
\end{proposition}

\textcolor{purple}{Since $P \lra \injt^{-1}(P)$ is an isomorphism for any projective $P$, we have }

\begin{proposition}
 \textcolor{purple}{Without any assumptions on $\Lambda$, $L_{0}\injt^{-1} = \mathbf{1}$ and $L_{i}\injt^{-1} = 0$ for any integer $i \geq 1$. \qed}
\end{proposition}

\textcolor{purple}{Having determined the derived functors of torsion, we turn our attention to cotorsion. It turns out it suffices to assume that $I$ be finitely generated rather than finitely presented.}

\begin{definition}
 \textcolor{purple}{A ring $\Lambda$ is said to be left FGE if the injective envelope of  $\Lambda$ viewed as a left module over itself is finitely generated.}
\end{definition}

\textcolor{purple}{Assume now that $\Lambda$ is left FGE, i.e., $I$ is finitely generated.
In this case, $(I, \blank)$ commutes with coproducts and 
therefore, by~\cite[Lemma 3.4]{MR-1}, the natural transformation 
$\tau_{(I, \blank)} : I^{\ast} \otimes \blank \lra (I, \blank)$ evaluates to an isomorphism on projectives. The diagram~\eqref{Eq:q-as-quot-functor} shows that the same is true for 
$\tau_{\injc} : \injc(\prescript{}{\Lambda}{\Lambda}) \otimes \blank \lra \injc$. Since $\tau_{\injc} : \injc(\prescript{}{\Lambda}{\Lambda}) \otimes \blank$ is right-exact, we have proved}

\begin{theorem}
\textcolor{purple}{ If $\Lambda$ is left FGE, then  
 $\tau_{\injc} : \injc(\prescript{}{\Lambda}{\Lambda}) \otimes \blank \lra \injc$ is the zeroth left derived functor of $\injc$ on left $\Lambda$-modules.  Consequently,  
\[
L_{i}\injc \simeq \Tor_{i}(\injc(\prescript{}{\Lambda}{\Lambda}), \blank)
\]
for all integers $i$. \qed}
\end{theorem}

\begin{remark}
\textcolor{purple}{The just proved theorem allows to compute the cotorsion of any projective left $\Lambda$-module $P$ over a left FGE ring $\Lambda$ by utilizing the isomorphism 
$\injc(\prescript{}{\Lambda}{\Lambda}) \otimes P \simeq \injc(P) $.}
\end{remark}

\textcolor{purple}{Since $\injc$ is determined by the short exact sequence~\eqref{Eq:def-L-I}, we have}

\begin{proposition}\label{P:w(q)=0}
 \textcolor{purple}{Without any assumptions on $\Lambda$, $w(\injc) = 0$. \qed}
\end{proposition}

\textcolor{purple}{As the top row of the diagram~\eqref{Eq:tau-q-1} evaluates to a short exact sequence on any projective, $\tau_{\injc^{-1}}$ evaluates to an isomorphism on any projective. And since  $Tr(\mathscr{I}, \prescript{}{\Lambda}{\Lambda}) \otimes \blank$ is right-exact, we have proved}

\begin{theorem}
 \textcolor{purple}{If $\Lambda$ is left FGE, then $\tau_{\injc^{-1}} : Tr(\mathscr{I}, \prescript{}{\Lambda}{\Lambda}) \otimes \blank \lra \injc^{-1}$ is the zeroth left derived functor of $\injc^{-1}$ on left $\Lambda$-modules. Consequently,
 \[
 L_{i}\injc^{-1} \simeq \Tor_{i}(Tr(\mathscr{I}, \prescript{}{\Lambda}{\Lambda}),\blank)
 \]
 for all integers $i$. \qed}
\end{theorem}

\textcolor{purple}{The short exact sequence $0 \lra \injc^{-1} \lra \mathbf{1} \lra \injc \lra 0$ shows, since $\injc$ is finitely presented, that  $\injc^{-1}$ is finitely presented too. To compute its defect we recall that~$w$ is an exact contravariant functor. Applying it to the above sequence we have, in view of Proposition~\ref{P:w(q)=0},}

\begin{proposition}
 \textcolor{purple}{Without any assumptions on $\Lambda$, $w(\injc^{-1}) \simeq 
 \prescript{}{\Lambda}{\Lambda}$. \qed}
\end{proposition}

\textcolor{purple}{To determine the right derived functors of $\injc$ and $\injc^{-1}$, we don't need any assumptions on $\Lambda$. Since injectives are cotorsion-free, we have}

\begin{proposition}
 \textcolor{purple}{Without any assumptions on $\Lambda$, $R^{i}\injc = 0$ for all integers $i$. \qed}
\end{proposition}

\textcolor{purple}{Since $\injc^{-1}(J) \lra J$ is an isomorphism for any injective $J$, we have} 

\begin{proposition}
 \textcolor{purple}{Without any assumptions on $\Lambda$, $R^{0}\injc^{-1} = \mathbf{1}$ and $R^{i}\injc^{-1} = 0$ for any integer $i \geq 1$. \qed}
\end{proposition}

\textcolor{purple}{By now the reader has undoubtedly noticed a rather intriguing parallelism between the properties of torsion and cotorsion. In the case of an FPE ring, we can add one more observation.}

\begin{theorem}
\textcolor{purple}{If $\Lambda$ is left FPE}, then the notions of torsion and cotorsion are dual.  More precisely, the right adjoint
\[
D_R:(\mod(\Lambda),\ab) \lra \fp(\Mod(\Lambda^{op}),\ab)
\]
of $D_A$ carries the injective torsion functor on right modules to the cotorsion functor on left modules, i.e., 
\[
D_R(\injt) \simeq \injc.
\]
\end{theorem}

\begin{proof}\textcolor{purple}{We claim that each term in the defining exact sequence 
\[
0 \lra \injt \lra \blank\otimes \Lambda \lra \blank\otimes I
\]
belongs to $\fp(\mod(\Lambda),\ab)$. This is clear for the last two terms because both~$\Lambda$ and~$I$ are finitely presented. To check this for $\injt$, we look at the full additive subcategory $\mathscr{P}$ of 
$\fp(\mod(\Lambda),\ab)$ determined by the representable functors. Given a natural transformation $(M_{2}, \blank) \to (M_{3}, \blank)$ with the $M_{i}$ finitely presented modules, there is, by Yoneda's lemma, a morphism $f : M_{3} \to M_{2}$ that determines that transformation. Since the $M_{i}$ are finitely presented, so is the cokernel $M_{1}$ of $f$. By the left-exactness of Hom, we have an exact sequence 
\[
0 \lra (M_{1}, \blank) \lra (M_{2}, \blank) \lra (M_{3}, \blank),
\]
which shows that $\mathscr{P}$ is closed under kernels. It now follows from~\cite[Proposition 2.1, b)]{A66} that the same is true for 
$\fp(\mod(\Lambda),\ab)$. This implies that $\injt$ belongs to this category, as was claimed. Since $\fp(\mod(\Lambda),\ab)$ is obviously closed under cokernels, it is abelian. As $D_R$ restricts to the AGJ duality there, that restriction must be exact, and therefore} we have an exact sequence 
\[
D_R(\blank\otimes I) \lra D_R(\blank\otimes\Lambda) \lra D_R(\injt) \lra 0.
\]
Because $I$ and $\Lambda$ are finitely presented modules, this yields the exact sequence 
\[
(I,\blank) \lra (\Lambda, \blank) \lra D_R(\injt) \lra 0,
\] 
which shows that $D_R(\injt) \simeq \injc$.
\end{proof}

\begin{corollary}
Let $\Lambda$ be an artin algebra. Then $D_R(\injt) \simeq \injc$, where $\injt$ and $\injc$ have arbitrary opposite chiralities.  \qed
\end{corollary}

\textcolor{purple}{Finally, we want to show that the FPE assumption allows to produce an analog for $\injc$ of Proposition~\ref{P:s-of-pex-is-ex} for $\injt$.}

\begin{proposition}
\textcolor{purple}{Suppose that $\prescript{}{\Lambda}{\Lambda}$ is left FPE  and let $0 \to C' \to C \to C'' \to 0$ be a pure exact sequence of left 
$\Lambda$-modules. Then the induced sequence 
\[
0 \lra \injc(C') \lra \injc(C) \lra \injc(C'') \lra 0
\]
is exact.} 
\end{proposition}

\begin{proof}
\textcolor{purple}{Since $I$ is finitely presented, we have a finite presentation $(I, \blank) \to (\Lambda, \blank) \to  \injc \to 0$ of $\injc$ with finitely presented arguments, which shows that $\injc$ preserves filtered colimits. A pure exact sequence is a directed colimit of split exact sequences. Thus, the sequence in question is a filtered colimit of similar sequences arising from split exact sequences. Each such sequence is exact because~$\injc$ is an additive functor. The result now follows since filtered colimits preserve exactness.}
\smallskip
\newline\textit{Second Proof.}
\textcolor{purple}{Assuming again that $I$ is finitely presented, we have that all terms in the short exact sequence $0 \to \Lambda \to I \to \Sigma \Lambda \to 0$ are finitely presented. Therefore, mapping each term into the given pure exact sequence, we have three exact sequences (\cite[Theorem 4.89]{Lam-99}) which assemble into a four-row and three-column diagram whose fourth row consists of the induced maps on the cotorsion modules. By the left-exactness of Hom, the columns are also exact and the desired result now follows from the snake lemma.}
\end{proof}

\begin{corollary}
\textcolor{purple}{If $\Lambda$ is left FPE, then the class of cotorsion-free left $\Lambda$-modules is closed under pure extensions. \qed}
\end{corollary}

\end{document}